\documentclass[a4paper]{article}
\usepackage[utf8]{inputenc}
\usepackage[top=1.45in, bottom=1.65in, left=1.55in, right=1.55in]{geometry}

\usepackage{amsmath, amsthm, amssymb, bbm, mathrsfs, graphicx}

\newtheorem{thm}{Theorem}
\newtheorem{lem}{Lemma}
\newtheorem{cor}{Corollary}
\newtheorem{conj}{Conjecture}
\newtheorem{qu}{Question}

\theoremstyle{definition}

\theoremstyle{definition}

\theoremstyle{remark}
\newtheorem*{remark}{Remark}

\newcommand{\proj}{\textrm{\normalfont{proj}}}
\newcommand{\ball}{\mathcal{B}}
\newcommand{\scap}{\textrm{\normalfont{cap}}}
\newcommand{\Scap}{\textrm{\normalfont{Cap}}}
\newcommand{\SO}{\textrm{\normalfont{O}}}
\newcommand{\Stab}{\textrm{\normalfont{Stab}}}
\newcommand{\diam}{\textrm{\normalfont{diam}}\,}
\newcommand{\supp}{\textrm{\normalfont{supp}}\,}
\newcommand{\m}{\mathbf{m}}
\newcommand{\meas}{\textrm{\normalfont{vol}}}
\newcommand{\dist}{\textrm{\normalfont{dist}}}
\newcommand{\Qcal}{\mathcal{Q}}
\newcommand{\R}{\mathbb{R}}
\renewcommand{\S}{\mathbb{S}}

\title{Geometrical sets with forbidden configurations}
\author{Davi Castro-Silva}
\date{\today}

\newcommand{\Addresses}{{
    \bigskip
    \small
    
    Davi de Castro Silva, \textsc{CWI \& QuSoft, Science Park 123, 1098 XG Amsterdam, The Netherlands}\par\nopagebreak
    \textit{E-mail address:} \texttt{davisilva15@gmail.com}
}}

\begin{document}

\maketitle

\begin{abstract}
    Given finite configurations $P_1, \dots, P_n \subset \R^d$, let us denote by $\m_{\R^d}(P_1, \dots, P_n)$ the maximum density a set $A \subseteq \R^d$ can have without containing congruent copies of any $P_i$.
    We will initiate the study of this geometrical parameter, called the \emph{independence density} of the considered configurations, and give several results we believe are interesting.
    For instance we show that, under suitable size and non-degeneracy conditions, $\m_{\R^d}(t_1 P_1, t_2 P_2, \dots, t_n P_n)$
    progressively `untangles' and tends to $\prod_{i=1}^n \m_{\R^d}(P_i)$ as the ratios $t_{i+1}/t_i$ between consecutive dilation parameters grow large;
    this shows an exponential decay on the density when forbidding multiple dilates of a given configuration, and gives a common generalization of theorems by Bourgain and by Bukh in geometric Ramsey theory.
    We also consider the analogous parameter $\m_{\S^d}(P_1, \dots, P_n)$ in the more complicated framework of sets on the unit sphere $\S^d$, obtaining the corresponding results in this setting.
\end{abstract}

\section{Introduction}

The general problem we consider in this paper can be phrased by the following question:
how large can a set be if it does not contain a given geometrical configuration?

The simplest and most well-studied instance of this problem concerns forbidden configurations of only two points on $\R^d$,
which are then characterized by their distance;
since there clearly exist unbounded sets on $\R^d$ which do not span a given distance, the appropriate notion of `largeness' must take into account their density rather than their cardinality or measure.
Define the \emph{upper density} $\overline{d}(A)$ of a measurable set $A \subseteq \R^d$ by
$$\overline{d}(A) = \limsup_{T \rightarrow \infty} \frac{\meas(A \cap [-T,\, T]^d)}{\meas([-T,\, T]^d)},$$
where $\meas$ denote the Lebesgue measure.
Our general problem in this case becomes:
what is the maximum upper density that a subset of $\R^d$ can have if it does not contain pairs of points at distance $1$?\footnote{Note that this problem is dilation invariant, so there is no loss of generality in assuming the forbidden distance to be $1$.}

This extremal density is commonly denoted $m_1(\R^d)$, and it is associated to the measurable chromatic number\footnote{The measurable chromatic number of $\R^d$ is the minimum number of measurable sets needed to partition $\R^d$ so that no two points belonging to the same part are at distance $1$ from each other.}
$\chi_m(\R^d)$ of the Euclidean space by the simple inequality $\chi_m(\R^d) \geq 1/m_1(\R^d)$.
Indeed, if no colour class contains pairs of points at unit distance, then each of them has upper density at most $m_1(\R^d)$, and it takes at least $1/m_1(\R^d)$ such classes to cover the whole space.
The parameter $m_1(\R^d)$ is many times studied in the context of providing lower bounds for the measurable chromatic number.

Despite significant research on the subject, there is still no dimension $d \geq 2$ for which the value of $m_1(\R^d)$ is known.
As far back as 1982, Erd\H{o}s~\cite{ErdosProblems} conjectured that $m_1(\R^2) < 1/4$, implying that any measurable planar set covering one fourth of the Euclidean plane contains pairs of points at unit distance;
this conjecture is still open.
A celebrated theorem of Frankl and Wilson~\cite{IntersectionTheorems} implies that $m_1(\R^d)$ decays exponentially with the dimension, and obtains the asymptotic upper bound $m_1(\R^d) \leq (1.2 + o(1))^{-d}$.
We refer the reader to Bachoc, Passuello and Thiery~\cite{densityBachoc} and to DeCorte, Oliveira and Vallentin~\cite{completelyPositive} for the best known bounds on $m_1(\R^d)$ and $\chi_m(\R^d)$.

The situation becomes even more complex and interesting when one forbids multiple distances $r_1, \dots, r_n > 0$;
let us denote by $\m_{\R^d}(r_1, \dots, r_n)$ the maximum upper density of a set in $\R^d$ avoiding all of these distances.
This parameter was first studied by Sz\'ekely~\cite{RemarksChromatic, Szekely} in connection with the chromatic number of geometric graphs, and it depends not only on the dimension of the space and number of forbidden distances but also on how these distances relate to each other.

In his first paper, Sz\'ekely pondered on the connection between the structure of a set of forbidden distances and the maximum density of a set in Euclidean space which avoids them all, and conjectured that $\m_{\R^2}\big((r_j)_{j \geq 1}\big) = 0$ whenever the sequence $(r_j)_{j \geq 1}$ of forbidden distances is unbounded.
His conjecture was proven by Furstenberg, Katznelson and Weiss~\cite{FurstKatzWeiss} using methods from ergodic theory, who obtained the following result:

\begin{thm}
If $A \subseteq \R^2$ has positive upper density, then there is some number $t_0$ such that for any $t \geq t_0$ one can find a pair of points $x, y \in A$ with $\|x - y\| = t$.
\end{thm}

Using Fourier analytic methods, Bourgain~\cite{Bourgain} was then able to generalize this theorem from two-point configurations on $\R^2$ to $d$-point configurations in general position on $\R^d$, for any $d \geq 2$.
For convenience, we shall say that a configuration $P \subset \R^d$ is \emph{admissible} if it has at most $d$ points and spans a $(|P|-1)$-dimensional affine hyperplane.
Bourgain showed:

\begin{thm} \label{thmbourgain}
Suppose $P \subset \R^d$ is admissible.
If $A \subseteq \R^d$ has positive upper density, then there is some number $t_0 > 0$ such that $A$ contains a congruent copy of $t \!\cdot\! P$ for all $t \geq t_0$.
\end{thm}

This result motivates the introduction of the \emph{independence density} of a given family of configurations $P_1, P_2, \dots, P_n \subset \R^d$, denoted $\m_{\R^d}(P_1,\, P_2,\, \dots,\, P_n)$, as the maximum upper density of a set in $\R^d$ which does not contain a congruent copy of any of these configurations.
This parameter generalizes our earlier notion of extremal density $\m_{\R^d}(r_1, \dots, r_n)$ from two-point to higher-order configurations, and can be seen as the natural analogue of the independence number\footnote{Given some finite hypergraph $H$, its independence number is the maximum size of a subset of vertices which does not entirely contain any edge.
Its independence density can then be defined as the independence number divided by the total number of vertices.}
for the (infinite) geometrical hypergraph on $\R^d$ whose edges are all isometric copies of $P_j$, $1 \leq j \leq n$.

With the notation now introduced, Bourgain's Theorem can be restated as the assertion that $\m_{\R^d}\big((t_j P)_{j \geq 1}\big) = 0$ for all admissible $P \subset \R^d$ and all unbounded positive sequences $(t_j)_{j \geq 1}$;
his proof in fact implies the stronger result that
$$\m_{\R^d}(t_1 P,\, t_2 P,\, \dots,\, t_n P) \rightarrow 0 \hspace{3mm} \text{as } n \rightarrow \infty$$
whenever the dilation parameters $t_j$ grow without bound.
Seen in this light, his results might inspire several further natural questions;
for instance:
\begin{itemize}
    \item[(Q1)] What is the rate of decay of $\m_{\R^d}(t_1 P,\, t_2 P,\, \dots,\, t_n P)$ with $n$ as the ratios $t_{j+1}/t_j$ between consecutive scales get large?
    \item[(Q2)] What possible values can be taken by the independence density
    $\m_{\R^d}(t_1 P,\, t_2 P,$ $\dots,\, t_n P)$
    of $n$ distinct dilates of a given configuration $P$?
    \item[(Q3)] Are there analogous results which are valid for other (non-Euclidean) spaces?
\end{itemize}
The goal of the present paper is to initiate the study of the independence density function $\m_{\R^d}$ and related geometrical parameters,
and the investigation of these three problems will serve as the driving force behind our analysis.

\subsection{Outline of the paper}

In Section~\ref{Space} we will formally define the independence density of a family of configurations, both in the entire space $\R^d$ and when restricted to bounded cubes in $\R^d$, and start our study of this geometrical parameter.
The methods we use are a mix of Fourier analysis, functional analysis and combinatorics.
The Fourier-analytic part is based mainly on Bourgain's arguments from~\cite{Bourgain}, and the combinatorial part is based on Bukh's arguments from~\cite{Bukh}
(where he considered similar problems to ours but concerning forbidden distances).
We do not assume that the reader is familiar with either of these papers, instead giving a presentation of the relevant parts of their reasoning that will be important to us.

The main tools to be used in this section will be a Counting Lemma (Lemma~\ref{lem:counting_space}) and a Supersaturation Theorem (Theorem~\ref{thm:supersat_space}), both of which are conceptually similar to results of the same name in graph and hypergraph theory (see~\cite{RegularityGraphs, quasirandomness, Supersaturation}).
Intuitively, the Counting Lemma says that the count of admissible configurations inside a given set does not significantly change if we blur the set a little;
this will be proven by Fourier-analytic methods.
The Supersaturation Theorem states that any bounded set $A \subseteq [-R,\, R]^d$, which is just slightly denser than the independence density of an admissible configuration $P$, must necessarily contain a positive proportion of all congruent copies of $P$ lying in $[-R,\, R]^d$;
this is proven by functional-analytic methods, via a compactness and weak$^*$ continuity argument.

We will then use these tools to obtain several results on the independence density parameter, and in particular answer questions (Q1) and (Q2) in the case where the considered configuration $P$ is admissible.
Regarding question (Q1), we show that $\m_{\R^d}(t_1 P,\, t_2 P,\, \dots,\, t_n P)$ tends to $\m_{\R^d}(P)^n$ as the ratios $t_{j+1}/t_j$ get large;
this generalizes a theorem of Bukh from two-point configurations to $k$-point configurations with $k \leq d$, and easily implies Bourgain's Theorem discussed in the Introduction.
As for question (Q2) we show that, by forbidding $n$ distinct dilates of such a configuration $P$, we can obtain as independence density any real number
strictly\footnote{Whether these boundary values can be attained is not yet clear.}
between $\m_{\R^d}(P)^n$ and $\m_{\R^d}(P)$, but none smaller than $\m_{\R^d}(P)^n$ or larger than $\m_{\R^d}(P)$.
We also prove:
\begin{itemize}
    \item[-] The general lower bound $\m_{\R^d}(P_1,\, P_2,\, \dots,\, P_n) \geq \prod_{i=1}^n \m_{\R^d}(P_i)$, which holds for all configurations $P_1, P_2, \dots, P_n \subset \R^d$;
    \item[-] Continuity of the independence density function $\m_{\R^d}$ on the set of admissible configurations; and
    \item[-] Existence of extremizer measurable sets (i.e. having maximal density) which avoid admissible configurations.
\end{itemize}

In Section~\ref{Sphere} we will consider these same questions but related to the more complicated setting of sets on the unit sphere $\S^d$.
We will also present (and prove) a spherical analogue of Bourgain's Theorem;
this is in line with our question~(Q3), as the sphere is the most well-studied non-Euclidean space.

Many of the arguments from the Euclidean setting will be used again in the spherical setting
(in particular the reliance on our two main combinatorial tools),
but there are also some complications we need to solve that are intrinsic to the sphere.
One of them is that harmonic analysis is (for our purposes) much more complicated on~$\S^d$ than it is on $\R^d$, which makes our proof of the spherical Counting Lemma correspondingly harder and more technical than its Euclidean counterpart.
Moreover, due to the lack of dilation invariance in the spherical setting, we will only be able to make a modest progress towards answering its analogue of question~(Q2)
(and the answer to question~(Q1) will be somewhat more intricate).
The other results proven in the Euclidean space setting will continue to hold in the same form for sets on the sphere.

Finally, in Section~\ref{Conclusion} we discuss some related results in the literature and suggest several intriguing open problems in line with the results presented here.

\subsection{Some remarks on notation}

The same denomination will be used for both a set and its indicator function;
for instance, if we are given $A \subseteq \R^d$, then $A(x) = 1$ if $x \in A$ and $A(x) = 0$ otherwise.
The group of permutations of $\{1, \dots, k\}$ is denoted by $\mathfrak{S}_k$.
Given a group $G$ acting on some space $X$ and an element $x$ of this space, we write $\Stab^G(x) := \{g \in G:\, g.x = x\}$ for the stabilizer subgroup of $x$.

The averaging notation $\mathbb{E}_{x \in X}$ is used to denote the expectation when the variable~$x$ is distributed uniformly over the set~$X$.
When~$X$ is (a subset of) a compact group~$G$, this measure is (the restriction of) the normalized Haar measure on~$G$, which is the unique Borel probability measure on~$G$ which is invariant by both left- and right-actions of this group.
Similarly, we write $\mathbb{P}_{x \in X}$ to denote the probability under this same distribution.

\section{Configurations in Euclidean space} \label{Space}

Throughout this section we shall fix an integer $d \geq 2$ and work on the $d$-dimensional Euclidean space $\R^d$, equipped with its usual inner product $x \cdot y$ and associated Euclidean norm $\|x\|$.
We denote by $\meas$ the Lebesgue measure on $\R^d$ and by $\mu$ the normalized Haar measure on the orthogonal group $\SO(\R^d) = \{O \in \R^{d \times d}:\, O^t O = I\}$.

Given $x \in \R^d$ and $R > 0$, we denote by $Q(x, R)$ the axis-parallel open cube of side length $R$ centered at $x$.
We write $d_{Q(x, R)}(A) := \meas(A \cap Q(x, R))/R^d$ for the density of $A \subseteq \R^d$ inside the cube $Q(x, R)$.
The upper density of a measurable set $A \subseteq \R^d$ can then be written as
$\overline{d}(A) = \limsup_{R \rightarrow \infty} d_{Q(0, R)}(A);$
if the limit exists, we shall instead denote it by $d(A)$.

A \emph{configuration} $P$ is just a finite subset of $\R^d$, and we define its \emph{diameter} $\diam P$ as the largest distance between two of its points.
Recall that a configuration $P \subset \R^d$ on $k$ points is said to be \emph{admissible} if $k \leq d$ and if $P$ is non-degenerate (that is, if it spans a $(k-1)$-dimensional affine hyperplane).
The space of $k$-point configurations can be given a metric induced from the Euclidean norm as follows:
if $P = \{v_1, \dots, v_k\}$ and $Q = \{u_1, \dots, u_k\}$, the distance between $P$ and $Q$ is
$$\|P - Q\|_{\infty} := \min_{\sigma \in \mathfrak{S}_k} \max_{1 \leq i \leq k} \|v_i - u_{\sigma(i)}\|,$$
where the minimum is taken over all permutations $\sigma$ of $\{1, \dots, k\}$.
It is easy to see that, under the topology induced by this metric, the set of admissible configurations is an open set and that it is dense inside the family of all subsets of $\R^d$ with at most $d$ elements.

We say that two configurations $P, Q \subset \R^d$ are \emph{congruent}, and write $P \simeq Q$, if they can be made equal using only rigid transformations;
that is, $P \simeq Q$ if and only if there exist $x\in \R^d$ and $T \in \SO(\R^d)$ such that $P = x + T \cdot Q$.
Given a configuration $P \subset \R^d$, we say that a set $A \subseteq \R^d$ \emph{avoids} $P$ if there is no subset of $A$ which is congruent to $P$.

We can now formally define our main object of study in this section, the \emph{independence density} of a configuration or family of configurations.
There are in fact two closely related versions of this parameter we will need, depending on whether we are considering bounded or unbounded configuration-avoiding sets.
Given $n \geq 1$ configuration $P_1, \dots, P_n \subset \R^d$, we then define the quantities
\begin{align*}
    \m_{\R^d}(P_1, \dots, P_n) &:= \sup \big\{\overline{d}(A):\, A \subset \R^d \text{ avoids } P_i,\, 1 \leq i \leq n\big\} \hspace{2mm} \text{and} \\
    \m_{Q(0, R)}(P_1, \dots, P_n) &:= \sup \big\{d_{Q(0, R)}(A):\, A \subset Q(0, R) \text{ avoids } P_i,\, 1 \leq i \leq n\big\}.
\end{align*}
These parameters are analogous to the notion of independence number of a hypergraph:
if we consider the hypergraph on vertex set $\R^d$ (resp. $Q(0, R)$) whose edges are all isometric copies of $P_j$, $1\leq j\leq n$, then $\m_{\R^d}(P_1, \dots, P_n)$ (resp. $\m_{Q(0, R)}(P_1, \dots, P_n)$) can be thought of as the density of a largest independent set in this hypergraph.

\begin{remark}
For the sake of clarity and notational convenience, whenever possible the results we give about independence density will be stated and proved in the case of only one forbidden configuration.
It can be easily verified that these results also hold in the case of several (but finitely many) forbidden configurations, with essentially unchanged proofs.
Whenever we need this greater generality we will mention how the corresponding statement would be in the case of several configurations.
\end{remark}

We start our investigations by proving a simple lemma which relates the two versions of independence density just defined:

\begin{lem} \label{lem:bounds_space}
For all configurations $P \subset \R^d$ and all $R > 0$, we have
$$\frac{\m_{Q(0, R)}(P)}{\left(1 + \frac{\diam P}{R}\right)^d} \leq \m_{\R^d}(P) \leq \m_{Q(0, R)}(P).$$
\end{lem}

\begin{proof}
For the first inequality, suppose $A \subseteq Q(0, R)$ is a set avoiding $P$ and consider the periodic set $A' := A + (R + \diam P) \mathbb{Z}^d$.
This set also avoids $P$, and it has density
$$d(A') = \frac {\meas(A)}{(R + \diam P)^d} = \frac{d_{Q(0, R)}(A)}{\left(1 + \frac{\diam P}{R}\right)^d}.$$
Since we can choose $d_{Q(0, R)}(A)$ arbitrarily close to $\m_{Q(0, R)}(P)$, the leftmost inequality follows.

Now let $A \subseteq \R^d$ be any set avoiding $P$, and note that $A \cap Q(x, R)$ also avoids $P$ for every $x \in \R^d$.
By fixing $\varepsilon > 0$ and then averaging over all $x$ inside a large enough cube $Q(0, R')$ (depending on $A$, $\diam P$ and $\varepsilon$), we conclude there is $x \in \R^d$ for which $\meas(A \cap Q(x, R)) > (\overline{d}(A) - \varepsilon) R^d$.
The rightmost inequality follows.
\end{proof}

As we are interested in the study of sets avoiding certain configurations, it is useful to also have a way of \emph{counting} how many such configurations there are in a given set.
For a given configuration $P = \{v_1, v_2, \dots, v_k\} \subset \R^d$ and a measurable set $A \subseteq \R^d$, we define
$$I_P(A) := \int_{\R^d} \int_{\SO(\R^d)} A(x + T v_1) A(x + T v_2) \cdots A(x + T v_{k}) \,d\mu(T) \,dx,$$
which represents how many (congruent) copies of $P$ are contained in $A$.
This quantity $I_P(A)$ can of course be infinite if the set $A$ is unbounded, but we will use it almost exclusively for bounded sets.
We can similarly define its weighted version
$$I_P(f) := \int_{\R^d} \int_{\SO(\R^d)} f(x + T v_1) f(x + T v_2) \cdots f(x + T v_{k}) \,d\mu(T) \,dx,$$
whenever $f: \R^d \rightarrow \R$ is a measurable function for which this integral makes sense
(say, for $f \in L^k(\R^d)$).
A large part of our analysis consists in getting a better understanding of the counting function $I_P$.


When a measurable set $A \subseteq \R^d$ avoids some configuration $P$, it is clear from the definition that $I_P(A) = 0$;
however, it is also possible for $I_P(A)$ to be zero even when $A$ contains congruent copies of $P$.
In intuitive terms, the condition $I_P(A) = 0$ means only that $A$ contains a negligible fraction of all possible copies of $P$.
The next result shows that this distinction is essentially irrelevant for most purposes:

\begin{lem}[Zero-measure removal] \label{lem:zero_meas_space}
Suppose $P \subset \R^d$ is a finite configuration and $A \subseteq \R^d$ is measurable.
If $I_P(A) = 0$, then we can remove a zero-measure subset of $A$ in order to remove all copies of $P$.
\end{lem}

\begin{proof}
By the Lebesgue Density Theorem, we have that
$$\lim_{\delta \rightarrow 0} \bigg| \frac{1}{\delta^d} \int_{Q(x, \delta)} A(y) \,dy - A(x) \bigg| = 0 \hspace{3mm} \text{for almost every } x \in \R^d.$$
Now we remove from $A$ all points $x$ for which this identity does not hold, thus obtaining a subset $B \subseteq A$ with $\meas(A \setminus B) = 0$ and
$$\lim_{\delta \rightarrow 0} \frac{1}{\delta^d} \int_{Q(x, \delta)} B(y) \,dy = 1 \hspace{3mm} \text{for all } x \in B.$$
We will show that no congruent copy of $P$ remains on this restricted set $B$.

Suppose for contradiction that $B$ contains a copy $\{u_1, \dots, u_k\}$ of $P$.
By assumption there exists some $\delta > 0$ such that
\begin{equation} \label{eq:dens_hyp}
    d_{Q(u_i, \delta)}(B) = \frac{1}{\delta^d} \int_{Q(u_i, \delta)} B(y) \,dy \geq 1 - \frac{1}{2^{d+1}k} \hspace{3mm} \text{for all } 1 \leq i \leq k;
\end{equation}
fix such a value of $\delta$.
Note that, if $d_{Q(x, \delta)}(B) \geq 1 - 1/(2^{d+1} k)$ for some $x \in \R^d$,
then for all $y \in Q(x, \delta/2)$ we have
\begin{align*}
    d_{Q(y, \delta/2)}(B)
    &= 1 - \frac{\meas\big(Q(y, \delta/2) \setminus B\big)}{(\delta/2)^d} \\
    &\geq 1 - \frac{\meas\big(Q(x, \delta) \setminus B\big)}{(\delta/2)^d} \\
    &= 1 - \frac{\delta^d \big(1 - d_{Q(x, \delta)}(B)\big)}{(\delta/2)^d} \\
    &\geq 1 - \frac{1}{2k}.
\end{align*}
Our hypothesis (\ref{eq:dens_hyp}) thus implies that $d_{Q(y, \delta/2)}(B) \geq 1 - 1/2k$ whenever $y \in Q(u_i, \delta/2)$ for some $1 \leq i \leq k$.

Let $\ell := \max \{\|u_i\|:\, 1 \leq i \leq k\}$ be the largest length of a vector in our copy of $P$, and let us write
$\ball(I,\, \delta/(4\ell)) := \big\{ T \in \SO(\R^d):\, \|T - I\| \leq \delta/(4\ell) \big\}$
for the ball of radius $\delta/(4\ell)$ in spectral norm centered on the identity $I$.
Note that, whenever $T \in \ball(I,\, \delta/(4\ell))$, we have that $T u_i \in Q(u_i, \delta/2)$ for each $1 \leq i \leq k$.
By the union bound we then have
\begin{align*}
    \int_{\R^d} \prod_{i = 1}^k B(x + Tu_i) \,dx
    &\geq \int_{Q(0, \delta/2)} \prod_{i = 1}^k B(x + Tu_i) \,dx \\
    &= \bigg(\frac{\delta}{2}\bigg)^d \mathbb{P}_{x \in Q(0, \delta/2)} \big(x + Tu_i \in B \text{ for all } 1 \leq i \leq k\big) \\
    &\geq \bigg(\frac{\delta}{2}\bigg)^d \bigg(1 - \sum_{i=1}^k \mathbb{P}_{x \in Q(0, \delta/2)}(x + Tu_i \notin B)\bigg) \\
    &= \bigg(\frac{\delta}{2}\bigg)^d \bigg(1 - \sum_{i=1}^k \big(1 - d_{Q(Tu_i, \delta/2)}(B)\big)\bigg) \\
    &\geq \frac{1}{2} \bigg(\frac{\delta}{2}\bigg)^d.
\end{align*}
This immediately implies that
$$I_P(B) \geq \int_{\R^d} \int_{\ball(I,\, \delta/(4\ell))} \prod_{i = 1}^k B(x + Tu_i) \,d\mu(T) \,dx \geq \frac{\mu\big(\ball(I,\, \delta/(4\ell))\big)}{2} \bigg(\frac{\delta}{2}\bigg)^d > 0,$$
contradicting our assumption that $I_P(A) = 0$ and finishing the proof.
\end{proof}

\subsection{Fourier analysis on $\R^d$ and the Counting Lemma} \label{FourierSection}

We next show that the count of copies of an admissible configuration $P$ inside a measurable set $A$ does not significantly change if we ignore its fine details and `blur' the set $A$ a little.
The philosophy is similar to the famous \emph{regularity method} in graph theory, where a large graph can be replaced by a much smaller weighted `reduced graph'
(which is an averaged version of the original graph which ignores its fine details)
without significantly changing the count of copies of any small subgraph.

The methods we will use are Fourier analytic in nature, drawing from Bourgain's arguments presented in~\cite{Bourgain}.
We define the Fourier transform on $\R^d$ by
$$\widehat{f}(\xi) := \int_{\R^d} f(x) e^{-2\pi i x\cdot \xi} \,dx \hspace{3mm} \text{and} \hspace{3mm} \widehat{\sigma}(\xi) := \int_{\R^d} e^{-2\pi i x\cdot \xi} \,d\sigma(x),$$
for a (complex-valued) function $f\in L^1(\R^d)$ and a finite Borel measure $\sigma$ on $\R^d$.
The convolution between two functions $f$, $g \in L^1(\R^d)$ is defined by
$$f * g(x) := \int_{\R^d} f(y) g(x-y) \,dy.$$
We recall the basic identities $\widehat{f * g}(\xi) = \widehat{f}(\xi) \widehat{g}(\xi)$ and
$$\int_{\R^d} f(x) \,d\sigma(x) = \int_{\R^d} \widehat{f}(\xi) \widehat{\sigma}(-\xi) \,d\xi,$$
as well as Parseval's Identity $\|f\|_2 = \|\widehat{f}\|_2$ for $f\in L^1(\R^d) \cap L^2(\R^d)$.
For background in Fourier analysis we refer the reader to the classic textbook of Stein and Weiss~\cite{SW71}.

Denote $\mathcal{Q}_{\delta}(x) := \delta^{-d} Q(0, \delta)(x)$.
This way, $f * \mathcal{Q}_{\delta}(x) = \delta^{-d} \int_{Q(x, \delta)} f(y) \,dy$ is the average of a function $f$ on the cube $Q(x, \delta)$.
Specializing to the indicator function of a set $A \subseteq \R^d$, we obtain $A * \mathcal{Q}_{\delta}(x) = d_{Q(x, \delta)}(A)$;
this represents a `blurring' of the set $A$ considered (see Figure~\ref{fig:setblur}).
What we wish to obtain is then an upper bound on the difference $|I_P(A) - I_P(A * \mathcal{Q}_{\delta})|$ which goes to zero as $\delta$ goes to zero, uniformly over all measurable sets $A \subseteq Q(0, R)$ (for any fixed $R > 0$).

\begin{figure}[ht]
    \centering
    \includegraphics[width = 0.9\textwidth]{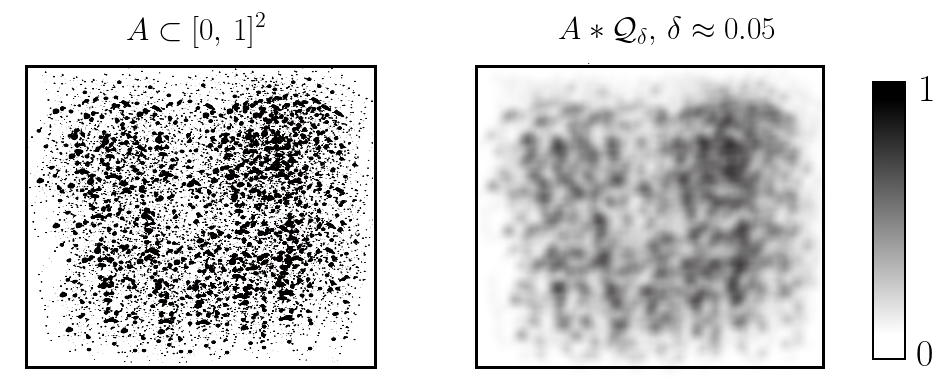}
    \caption{An example of a planar set $A$ on the unit square and the corresponding function $A * \Qcal_{\delta}$, for some small $\delta$;
    the shades of gray represent the value this function takes at each point.}
    \label{fig:setblur}
\end{figure}

Before delving into the details of our argument, let us present a simple telescoping sum argument which will be needed here and will be reused several times in this paper.
Suppose we wish to bound from above the expression
$$|I_P(f) - I_P(g)| = \bigg| \int_{\R^d} \int_{\SO(\R^d)} \bigg(\prod_{i=1}^k f(x + T v_i) - \prod_{i=1}^k g(x + T v_i)\bigg) \,d\mu(T) \,dx \bigg|$$
for some given functions $f$, $g$ and some configuration $P = \{v_1, \dots, v_k\}$.
Since we can rewrite the term inside the parenthesis above
as the telescoping sum
$$\sum_{i=1}^k \bigg(\prod_{j=1}^{i-1} f(x + T v_j)\bigg) \big(f(x + T v_i) - g(x + T v_i)\big) \bigg(\prod_{j=i+1}^{k} g(x + T v_j)\bigg),$$
it follows from the triangle inequality that $|I_P(f) - I_P(g)|$ is at most
$$\sum_{i=1}^k \bigg| \int_{\R^d} \int_{\SO(\R^d)} \prod_{j=1}^{i-1} f(x + T v_j) \big(f(x + T v_i) - g(x + T v_i)\big) \prod_{j=i+1}^{k} g(x + T v_j) \,d\mu(T) \,dx \bigg|.$$
To obtain some bound for $|I_P(f) - I_P(g)|,$ it then suffices to obtain a similar bound for an expression of the form
$$\bigg|\int_{\R^d} \int_{\SO(\R^d)} h_1(x + T u_1) \cdots h_{k-1}(x + T u_{k-1}) \big(f(x + T u_k) - g(x + T u_k)\big) \,d\mu(T) \,dx\bigg|$$
whenever each $h_i$ is either $f$ or $g$, and whenever $(u_1, \dots, u_k)$ is a permutation of the points of $P$.

We shall refer to an argument of this form (breaking a difference of products into a telescoping sum, using the triangle inequality and bounding each term of the resulting expression) as the \emph{telescoping sum trick}.
It is frequently used in modern graph and hypergraph theory when estimating the number of subgraphs inside a given large (hyper)graph $G$ with the aid of edge-discrepancy measures such as the cut norm;
such results are usually known as \emph{counting lemmas}, and are an essential part of the regularity method we have already mentioned
(see the surveys~\cite{RegularityGraphs, quasirandomness} for details).

\medskip

In our arguments we will also need some analytic facts and estimates, which we now provide.
Given an $m$-dimensional subspace $U\subseteq \R^d$, we denote by $\sigma_U^{(m-1)}$ the uniform probability measure on its unit sphere $\mathbb{S}_U^{m-1} := \big\{x\in U:\, \|x\| = 1\big\}$.
This measure is closely related to the Haar measure $\mu_U$ on the orthogonal group $\SO(U)$:
if $X \subseteq \S_U^{m-1}$ is a measurable set and $x\in \S_U^{m-1}$ is any point, then
$$\sigma_U^{(m-1)}(X) = \mu_U\big( \big\{T\in \SO(U):\, Tx \in X\big\} \big).$$
(See for instance \cite[Appendix~A.5]{harmAnalSphere} for a simple proof of this fact.)
Given $T\in \SO(\R^d)$, we write $TU := \{Tu:\, u \in U\}$ for the rotated subspace.

\begin{lem} \label{lem:Fourier_bound}
There are constants $C_1, C_2 > 0$ (depending on the dimension $d$) such that
$$|1 - \widehat{\mathcal{Q}}_{\delta}(\xi)| \leq C_1 \delta^2 \|\xi\|^2 \quad \text{for all $\delta>0$, $\xi\in \R^d$}$$
and, if $V$ is an $m$-dimensional subspace of $\R^d$,
$$\int_{\SO(\R^d)} |\widehat{\sigma}^{(m-1)}_{TV}(\xi)|^2 \,d\mu(T) \leq C_2 \|\xi\|^{-(m-1)} \quad \text{for all $\xi \in \R^d\setminus \{0\}$.}$$

\end{lem}

\begin{proof}
For the first inequality, note that
$$\widehat{\mathcal{Q}}_{\delta}(\xi) = \frac{1}{\delta^d} \int_{Q(0,\delta)} e^{-2\pi i x\cdot \xi} \,dx = \prod_{j=1}^d \frac{1}{\delta} \int_{-\delta/2}^{\delta/2} e^{-2\pi i x_j \xi_j} \,dx_j = \prod_{j=1}^d \frac{\sin(\pi\delta \xi_j)}{\pi\delta \xi_j}$$
(where the $j$-th term in the product is $1$ if $\xi_j = 0$).
It follows from the Taylor expansion of $\sin(\cdot)$ that $|x - \sin(x)| \leq C |x|^3$ for some $C>0$ and all $x\in [-1, 1]$.
As the sine function is bounded, we conclude there is some constant $C_1 > 0$ (depending on $d$) for which
$$|1 - \widehat{\mathcal{Q}}_{\delta}(\xi)| = \bigg|1 - \prod_{j=1}^d \frac{\sin(\pi\delta \xi_j)}{\pi\delta \xi_j}\bigg| \leq C_1 \sum_{j=1}^d (\delta \xi_j)^2 = C_1 \delta^2 \|\xi\|^2$$
holds for all $\delta>0$, $\xi\in \R^d$.

For the second inequality we use the estimate
\begin{equation*}
    \big|\widehat{\sigma}^{(m-1)}_{U}(\xi)\big| \leq K \|\pi_U \xi\|^{-(m-1)/2},
\end{equation*}
where $\pi_U \xi$ is the orthogonal projection of $\xi$ onto $U$ and $K$ is an absolute constant.
This estimate follows from
$$\widehat{\sigma}^{(m-1)}_{U}(\xi) = \int_{\R^d} e^{-2\pi i x\cdot \xi} \,d\sigma^{(m-1)}_{U}(x) = \int_{U} e^{-2\pi i x\cdot \pi_U\xi} \,d\sigma^{(m-1)}_{U}(x)$$
and the well-known asymptotic bound $|\widehat{\sigma}_{\R^m}^{(m-1)}(\xi)| = O(\|\xi\|^{-(m-1)/2})$ for the unit sphere on $\R^m$
(see for instance Chapter VIII, Section 3 in Stein's book~\cite{Stein93}).
For any $\xi\in \R^d\setminus \{0\}$, we then have that
\begin{align*}
    \int_{\SO(\R^d)} |\widehat{\sigma}^{(m-1)}_{TV}(\xi)|^2 \,d\mu(T)
    &\leq \int_{\SO(\R^d)} K^2 \|\pi_{TV} \xi\|^{-(m-1)} \,d\mu(T) \\
    &= K^2 \int_{\SO(\R^d)} \|\pi_{V} (T^{-1}\xi)\|^{-(m-1)} \,d\mu(T) \\
    &= K^2 \|\xi\|^{-(m-1)} \int_{\S^{d-1}} \|\pi_{\R^m} y\|^{-(m-1)} \,d\sigma^{(d-1)}_{\R^d}(y),
\end{align*}
where we performed the change of variables $y = T^{-1} \xi/\|\xi\|$.

It now suffices to show that the last integral above is finite, which we will do by induction on $d \geq m$.
In the base case where $d=m$ the integral is clearly equal to $1$, since the projection operator~$\pi_{\R^m}$ is the identity.
If $d \geq m+1$, parameterize $\S^{d-1}$ by
$$y = (z \sin\theta,\, \cos\theta) \quad \text{for $z\in \S^{d-2}$, $\theta \in [0, \pi]$;}$$
denoting by $\omega_{d-1}$ (resp. $\omega_{d-2}$) the total Lebesgue measure of the unit sphere of $\R^d$ (resp. $\R^{d-1}$), this change of variables gives
$$\omega_{d-1} \,d\sigma^{(d-1)}_{\R^d}(y) = \omega_{d-2} (\sin\theta)^{d-2} \,d\sigma^{(d-2)}_{\R^{d-1}}(z) \,d\theta.$$
We then obtain
\begin{align*}
    \int_{\S^{d-1}} &\|\pi_{\R^m} y\|^{-(m-1)} \,d\sigma^{(d-1)}_{\R^d}(y) \\
    &= \frac{1}{\omega_{d-1}} \int_{0}^{\pi} \int_{\S^{d-2}} \big(\sin\theta \,\|\pi_{\R^m} z\|\big)^{-(m-1)} \omega_{d-2} (\sin\theta)^{d-2} \,d\sigma^{(d-2)}_{\R^{d-1}}(z) \,d\theta \\
    &= \frac{\omega_{d-2}}{\omega_{d-1}} \bigg(\int_{0}^{\pi} (\sin\theta)^{d-m-1} \,d\theta\bigg) \int_{\S^{d-2}} \|\pi_{\R^m} z\|^{-(m-1)} \,d\sigma^{(d-2)}_{\R^{d-1}}(z) \\
    &\leq \frac{\omega_{d-2} \pi}{\omega_{d-1}} \int_{\S^{d-2}} \|\pi_{\R^m} z\|^{-(m-1)} \,d\sigma^{(d-2)}_{\R^{d-1}}(z),
\end{align*}
and the desired bound follows by induction.
\end{proof}

We are now ready to formally state and prove our main technical tool in the Euclidean setting, which by analogy with methods from graph theory we shall call the Counting Lemma.
We note that the main steps of its proof were already present in Bourgain's paper~\cite{Bourgain}.

\begin{lem} [Counting Lemma] \label{lem:counting_space}
For every admissible configuration $P \subset \R^d$ there exists a constant $C_P > 0$ such that the following holds:
for every $R > 0$ and any measurable set $A \subseteq Q(0, R)$, we have that
$$|I_P(A) - I_P(A * \mathcal{Q}_{\delta})| \leq C_P \delta^{1/4} R^d \quad \text{for all $\delta \in (0, 1].$}$$
Moreover, the same constant $C_P$ can be made to hold uniformly over all configurations $P'$ inside a neighborhood of $P$.
\end{lem}

\begin{proof}
Let $(v_1, \dots, v_k)$ be a fixed permutation of the points of $P$.
We will work a bit more generally and show that a bound as in the statement of the lemma holds for
$$\bigg| \int_{\R^d} \int_{\SO(\R^d)} f_1(x + T v_1) \cdots f_{k-1}(x + T v_{k-1}) \big( f_k(x + T v_k) - f_k * \mathcal{Q}_{\delta}(x + T v_k) \big) \,d\mu(T) \,dx \bigg|$$
whenever $f_1,\, \dots,\, f_k:\, Q(0, R) \rightarrow [-1, 1]$ are measurable functions.
By our telescoping sum trick, this immediately implies the result.

We first exploit the translation invariance of the problem in order to simplify the argument later on.
Let $U \subset \R^d$ denote the $(k-2)$-dimensional affine hyperplane spanned by $v_1, \dots, v_{k-1}$, and let $\pi_U v_k$ be the orthogonal projection of $v_k$ onto $U$
(so $\pi_U v_k$ is the point in $U$ which is closest to $v_k$).
By translating all points in $P$ by $-\pi_U v_k$, we may assume that $U$ contains the origin (being thus a subspace of $\R^d$) and that $v_k$ belongs to its orthogonal complement $U^{\perp}$.
Note that $v_k \neq 0$ since the points in $P$ are affinely independent, and $U^{\perp}$ has dimension $d-k+2 \geq 2$;
these are the two properties we will need in the proof which require the assumption that $P$ is admissible.

Let $H := \Stab^{\SO(\R^d)}(U)$ denote the subgroup of orthogonal transformations which act trivially on the subspace $U$, and let $\nu_H$ be the Haar measure on $H$.
Let $G := f_k - f_k * \mathcal{Q}_{\delta}$ and, for a given $T\in \SO(\R^d)$, define the function $F_T:\, Q(0, R) \rightarrow [-1, 1]$ by
$$F_T(x) = \prod_{i=1}^{k-1} f_i(x + Tv_i).$$
The integrand on the expression we wish to bound can then be written more succinctly as $F_T(x) G(x + T v_k)$.
By symmetry of the Haar measure $\mu$ we conclude that
\begin{align*}
    \int_{\SO(\R^d)} F_T(x) \, G(x + T v_k) \,d\mu(T)
    &= \int_{H} \int_{\SO(\R^d)} F_{TS}(x) \, G(x + TS v_k) \,d\mu(T) \,d\nu_H(S) \\
    &= \int_{\SO(\R^d)} \int_{H} F_T(x) \, G(x + TS v_k) \,d\nu_H(S) \,d\mu(T),
\end{align*}
where we have used that $F_{TS} = F_T$ for all $S\in H$, since by definition $Sv_i = v_i$ for all $1\leq i\leq k-1$.
Using this identity, we conclude that the expression we wish to bound is at most
\begin{equation} \label{whattoboundspace}
    \bigg|\int_{\SO(\R^d)} \bigg(\int_{\R^d} \int_{H} F_T(x) \, G(x + TS v_k) \,d\nu_H(S) \,dx\bigg) d\mu(T) \bigg|.
\end{equation}

Now we concentrate on the expression inside the parenthesis in \eqref{whattoboundspace}, for some fixed $T\in \SO(\R^d)$.
We claim that, when $S$ is distributed according to the Haar measure on $H$, the variable $y := TS(v_k/\|v_k\|)$ is uniformly distributed on the unit sphere of the subspace $TU^{\perp}$.
This follows from the fact that $v_k/\|v_k\|$ is on the unit sphere of $U^{\perp}$, and $H := \Stab^{\SO(\R^d)}(U)$ is
isomorphic\footnote{Every orthogonal transformation on  $U^{\perp}$ can be identified with an element of $\Stab^{\SO(\R^d)}(U)$ by tensoring with the identity on $U$, with this identification being bijective and measure-preserving.}
to the orthogonal group of $U^{\perp}$.
Denoting by $\sigma^{(d-k+1)}_{TU^{\perp}}$ the normalized surface measure on the unit sphere of $TU^{\perp}$, we can then write the expression inside the parenthesis in \eqref{whattoboundspace} as
\begin{align*}
    \int_{\R^d} \int_{\R^d} F_T(x) &\,G(x + \|v_k\|y) \,d\sigma^{(d-k+1)}_{TU^{\perp}}(y) \,dx \\
    &= \int_{\R^d} F_T(x) \int_{\R^d} e^{2\pi i x\cdot \xi} \,\widehat{G}(\xi) \,\widehat{\sigma}^{(d-k+1)}_{TU^{\perp}}(-\|v_k\|\xi) \,d\xi \,dx \\
    &= \int_{\R^d} \widehat{F}_T(-\xi) \,\widehat{G}(\xi) \,\widehat{\sigma}^{(d-k+1)}_{TU^{\perp}}(-\|v_k\|\xi) \,d\xi.
\end{align*}
Integrating over $T\in \SO(\R^d)$ and applying Cauchy-Schwarz to the inner integral, we conclude that~\eqref{whattoboundspace} is at most
\begin{align*}
    \int_{\SO(\R^d)} &\|\widehat{F}_T\|_2 \bigg( \int_{\R^d} |\widehat{G}(\xi)|^2 \,|\widehat{\sigma}^{(d-k+1)}_{TU^{\perp}}(\|v_k\|\xi)|^2 \,d\xi \bigg)^{1/2} d\mu(T) \\
    &= \int_{\SO(\R^d)} \|F_T\|_2 \bigg( \int_{\R^d} |\widehat{f}_k(\xi)|^2 \,|1 - \widehat{\mathcal{Q}}_{\delta}(\xi)|^2 \,|\widehat{\sigma}^{(d-k+1)}_{TU^{\perp}}(\|v_k\|\xi)|^2 \,d\xi \bigg)^{1/2} d\mu(T),
\end{align*}
where we have used Parseval's Identity and the convolution identity.

Since $|F_T(x)| \leq |f_1(x + Tv_1)|$ pointwise, it follows that $\|F_T\| \leq \|f_1\|_2$ for all $T \in \SO(\R^d)$.
Using this inequality and applying Cauchy-Schwarz to the outer integral, we see that the expression above is at most
\begin{equation} \label{eq:finalexp}
    \|f_1\|_2 \bigg( \int_{\SO(\R^d)} \int_{\R^d} |\widehat{f}_k(\xi)|^2 \,|1 - \widehat{\mathcal{Q}}_{\delta}(\xi)|^2 \,|\widehat{\sigma}^{(d-k+1)}_{TU^{\perp}}(\|v_k\|\xi)|^2 \,d\xi \,d\mu(T) \bigg)^{1/2}.
\end{equation}
Finally, the double integral in~\eqref{eq:finalexp} can be bounded using the Fourier estimates given in Lemma~\ref{lem:Fourier_bound}, as we now show.
Divide the integral over $\R^d$ into two parts, corresponding to the bounded region where $\|\xi\| \leq (\delta \|v_k\|)^{-1/2}$ and the unbounded region where $\|\xi\| > (\delta \|v_k\|)^{-1/2}$.
For the bounded region we note that $|\widehat{\sigma}^{(d-k+1)}_{TU^{\perp}}(\xi)| \leq 1$ for all $T\in \SO(\R^d)$, $\xi\in \R^d$, and use the first inequality in Lemma~\ref{lem:Fourier_bound} to obtain
\begin{align*}
    \int_{\SO(\R^d)} \int_{\|\xi\| \leq (\delta \|v_k\|)^{-1/2}} &|\widehat{f}_k(\xi)|^2 \,|1 - \widehat{\mathcal{Q}}_{\delta}(\xi)|^2 \,|\widehat{\sigma}^{(d-k+1)}_{TU^{\perp}}(\|v_k\|\xi)|^2 \,d\xi \,d\mu(T) \\
    &\leq \int_{\SO(\R^d)} \int_{\|\xi\| \leq (\delta \|v_k\|)^{-1/2}} |\widehat{f}_k(\xi)|^2 \big(C_1 \delta^2 \|\xi\|^2\big)^2 \,d\xi \,d\mu(T) \\
    &\leq C_1^2 \delta^2 \|v_k\|^{-2} \|f_k\|_2^2.
\end{align*}
For the unbounded region, we use the simple estimate $|\widehat{\mathcal{Q}}_{\delta}(\xi)| \leq \|\mathcal{Q}_{\delta}\|_1 = 1$ and the second inequality in Lemma~\ref{lem:Fourier_bound} to conclude that
\begin{align*}
    \int_{\SO(\R^d)} \int_{\|\xi\| > (\delta \|v_k\|)^{-1/2}} &|\widehat{f}_k(\xi)|^2 \,|1 - \widehat{\mathcal{Q}}_{\delta}(\xi)|^2 \,|\widehat{\sigma}^{(d-k+1)}_{TU^{\perp}}(\|v_k\|\xi)|^2 \,d\xi \,d\mu(T) \\
    &\leq \int_{\|\xi\| > (\delta \|v_k\|)^{-1/2}} 4 |\widehat{f}_k(\xi)|^2 \,\int_{\SO(\R^d)} |\widehat{\sigma}^{(d-k+1)}_{TU^{\perp}}(\|v_k\|\xi)|^2 \,d\mu(T)  \,d\xi \\
    &\leq 4 \left\|f_k\right\|_2^2 \sup_{\|\xi\| > (\delta \|v_k\|)^{-1/2}} \int_{\SO(\R^d)} |\widehat{\sigma}^{(d-k+1)}_{TU^{\perp}}(\|v_k\|\xi)|^2 \,d\mu(T) \\
    &\leq 4C_2 \,(\delta \|v_k\|^{-1})^{(d-k+1)/2} \|f_k\|_2^2.
\end{align*}

Summing the bounds obtained for both regions shows that, for $d \geq k$ and $0 < \delta \leq 1$, we can bound expression~\eqref{eq:finalexp} by
$$\big(C_1^2 \|v_k\|^{-2} + 4C_2 \|v_k\|^{-(d-k+1)/2}\big)^{1/2} \delta^{1/4} \|f_1\|_2 \|f_k\|_2,$$
and the inequality in the statement of the lemma follows.
Since this last bound depends continuously on the positioning of the points of~$P$ (which gives the value of~$\|v_k\|$), the claim that the obtained constant $C_P$ can be made uniform inside a neighborhood of $P$ also follows.
\end{proof}

We remark that the proof above is the only place (in the Euclidean setting) where we make explicit use of the assumption that a configuration is admissible.
However, as the Counting Lemma will be an essential ingredient of several later results, this requirement will be inherited by them as well.

\subsection{Continuity properties of the counting function}

Given some configuration $P$ on the space $\R^d$,
it is sometimes important to understand how much the count of congruent copies of $P$ on a set $A \subseteq \R^d$ can change if we perturb the set $A$ a little.
An instance of this problem was already considered in the Counting Lemma, where the perturbation was given by blurring and it was seen that the counting function $I_P$ is somewhat robust to small perturbations
(in the case of admissible configurations).

Using our telescoping sum trick, it is easy to show that $I_P$ is also robust to small perturbations measured by the $L^{\infty}$ norm;
more precisely, $I_P$ is continuous on $L^{\infty}(Q(0, R))$ for any fixed $R > 0$.
When $P$ is admissible, we obtain the following significantly stronger continuity property:

\begin{lem}[Weak$^*$ continuity] \label{lem:weak*cont_space}
If $P \subset \R^{d}$ is an admissible configuration, then for every fixed $R > 0$ the function $I_P$ is weak$^*$ continuous on the unit ball of $L^{\infty}(Q(0, R))$.
\end{lem}

\begin{proof}
Denote the closed unit ball of $L^{\infty}(Q(0, R))$ by $\ball_{\infty}$.
Since $\ball_{\infty}$ endowed with the weak$^*$ topology is metrizable (see e.g. \cite[Corollary~2.6.20]{Megginson98}), it suffices to prove that $I_P$ is sequentially continuous, i.e. that $I_P(f_i) \xrightarrow{i \rightarrow \infty} I_P(f)$ whenever $f_i \xrightarrow{i \rightarrow \infty} f$.

Suppose then $(f_i)_{i \geq 1} \subset \ball_{\infty}$ is a sequence weak$^*$ converging to $f \in \ball_{\infty}$.
It follows that, for every $x \in Q(0, R)$ and every $\delta > 0$, we have
\begin{equation*}
    f_i * \Qcal_{\delta}(x) = \delta^{-d} \int_{Q(x, \delta)} f_i(y) \,dy
    \,\xrightarrow{i \rightarrow \infty}\, \delta^{-d} \int_{Q(x, \delta)} f(y) \,dy = f * \Qcal_{\delta}(x).
\end{equation*}
Since $f * \Qcal_{\delta}$ and each $f_i * \Qcal_{\delta}$ are Lipschitz with the same constant (depending only on $\delta$, as $\|f\|_{\infty}$, $\|f_i\|_{\infty} \leq 1$) and $Q(0, R)$ is bounded, this easily implies that
$$\|f_i * \Qcal_{\delta} - f * \Qcal_{\delta}\|_{\infty} \rightarrow 0 \hspace{3mm} \text{as } i \rightarrow \infty.$$
In particular, it follows that $\lim_{i \rightarrow \infty} I_P(f_i * \Qcal_{\delta}) = I_P(f * \Qcal_{\delta})$.

Since $P$ is admissible, by the Counting Lemma (Lemma~\ref{lem:counting_space}) we have
$$|I_P(f * \Qcal_{\delta}) - I_P(f)|, \hspace{2mm} |I_P(f_i * \Qcal_{\delta}) - I_P(f_i)| \leq C_P \delta^{1/4} R^d \hspace{3mm} \text{for all } i \geq 1.$$
Choosing $i_0(\delta) \geq 1$ sufficiently large so that
$$|I_P(f_i * \Qcal_{\delta}) - I_P(f * \Qcal_{\delta})| \leq C_P \delta^{1/4} R^d \hspace{5mm} \text{for all } i \geq i_0(\delta),$$
we conclude that
\begin{align*}
    |I_P(f) - I_P(f_i)| &\leq |I_P(f) - I_P(f * \Qcal_{\delta})| + |I_P(f * \Qcal_{\delta}) - I_P(f_i * \Qcal_{\delta})| \\
    & \hspace{1cm} + |I_P(f_i * \Qcal_{\delta}) - I_P(f_i)| \\
    &\leq 3 C_P \delta^{1/4} R^d \hspace{5mm} \text{for all } i \geq i_0(\delta).
\end{align*}
Since $\delta > 0$ is arbitrary, this implies that $\lim_{i \rightarrow \infty} I_P(f_i) = I_P(f)$, as wished.
\end{proof}

We will also need an \emph{equicontinuity} property for the family of counting functions $P \mapsto I_P(A)$, over all bounded measurable sets $A \subseteq \R^d$.
In what follows we shall write $\ball(P, r) \subset (\R^d)^k$ for the ball of radius $r$
centered on $P = \{v_1, \dots, v_k\}$, where we recall that the distance from $P$ to $Q = \{u_1, \dots, u_k\}$ is given by
$$\|P - Q\|_{\infty} = \min_{\sigma \in \mathfrak{S}_k} \max_{1 \leq i \leq k} \|v_i - u_{\sigma(i)}\|.$$

\begin{lem}[Equicontinuity] \label{lem:equicont_space}
For every admissible $P \subset \R^d$ and every $\varepsilon > 0$ there is $\delta > 0$ such that the following holds:
if $P' \in \ball(P, \delta)$, then for all $R \geq 1$ we have
$$|I_{P'}(A) - I_P(A)| \leq \varepsilon R^d \hspace{3mm} \text{for all measurable } A \subseteq Q(0, R).$$
\end{lem}

\begin{proof}
We will use the fact that the constant $C_P$ promised in the Counting Lemma can be made uniform inside a small neighborhood of $P$;
more precisely, there is $r > 0$ and a constant $\tilde{C}_P > 0$ such that
$$|I_{P'}(A) - I_{P'}(A * \Qcal_{\rho})| \leq \tilde{C}_P \rho^{1/4} R^d \hspace{5mm} \text{for all } \rho \in (0, 1]$$
holds for all $P' \in \ball(P, r)$, $R > 0$ and (measurable) $A \subseteq Q(0, R)$.

Fix constants $R \geq 1$ and $\delta, \rho \in (0, 1]$ with $\delta < \rho$.
For any set $A \subseteq Q(0, R)$ and any points $x,\, y \in Q(0, R)$ with $\|x - y\| \leq \delta$, we have that
\begin{align*}
    |A * \Qcal_{\rho}(x) - A * \Qcal_{\rho}(y)|
    &= \frac{\big| \meas(A \cap Q(x,\, \rho)) - \meas(A \cap Q(y,\, \rho)) \big|}{\rho^d} \\
    &\leq \frac {\meas(Q(x,\, \rho) \setminus Q(y,\, \rho))}{\rho^d} \\
    &\leq \frac{\rho^d - (\rho - \delta)^d}{\rho^d}.
\end{align*}
Noting that $A * \Qcal_{\rho}$ is supported on $Q(0, R+\rho) \subseteq Q(0, 2R)$, we conclude from our telescoping sum trick that
$$|I_{P'}(A * \Qcal_{\rho}) - I_P(A * \Qcal_{\rho})| \leq k \frac{\rho^d - (\rho - \delta)^d}{\rho^d} (2R)^d$$
whenever $\|P' - P\|_{\infty} \leq \delta$.

Now take $\rho \in (0, 1]$ small enough so that $\tilde{C}_P \rho^{1/4} \leq \varepsilon/4$, and for this value of $\rho$ take $0 < \delta < \min\{r,\, \rho\}$ small enough so that
$$(\rho - \delta)^d \geq \Big(1 - \frac{\varepsilon}{2^{d+1} k}\Big) \rho^d.$$
Then, for any configuration $P' \in \ball(P,\, \delta)$ and any set $A \subseteq Q(0, R)$, we obtain
\begin{align*}
    |I_{P'}(A) - I_P(A)| &\leq |I_{P'}(A) - I_{P'}(A * \Qcal_{\rho})| + |I_{P'}(A * \Qcal_{\rho}) - I_P(A * \Qcal_{\rho})| \\
    & \hspace{1cm}
    + |I_P(A * \Qcal_{\rho}) - I_P(A)| \\
    &\leq \tilde{C}_P \rho^{1/4} R^d + k \frac{\rho^d - (\rho - \delta)^d}{\rho^d} (2R)^d + \tilde{C}_P \rho^{1/4} R^d \\
    &\leq \frac{\varepsilon}{4} R^d + k \frac{\varepsilon}{2^{d+1} k} (2R)^d + \frac{\varepsilon}{4} R^d
    = \varepsilon R^d,
\end{align*}
as desired.
\end{proof}

\subsection{The Supersaturation Theorem}

Now we wish to show that geometrical hypergraphs encoding copies of some admissible configuration $P$ have a nice \emph{supersaturation property}:
if a set $A \subseteq \R^d$ is just slightly denser than the independence density of $P$, then it must contain a positive proportion of all congruent copies of $P$.
This result is quite similar, both formally and in spirit, to an important combinatorial theorem of Erd\H{o}s and Simonovits~\cite{Supersaturation} in the setting of forbidden graphs and hypergraphs.

\begin{remark}
The insight that supersaturation results can be used to better study extremal geometrical problems of the kind we are interested in is due to Bukh~\cite{Bukh}.
He introduced the notion of a `supersaturable property' as any characteristic of measurable sets which satisfies several conditions meant to enable the proof of a supersaturation result;
the prototypical and most important example of supersaturable property given in Bukh's paper is that of avoiding a finite collection of distances.
Here we will obtain similar results in the case of avoiding general admissible configurations, but our method of proof is more analytical in nature and quite different from his.
\end{remark}

Using our zero-measure removal lemma (Lemma~\ref{lem:zero_meas_space}), we can immediately obtain a weak supersaturation property which holds for any $R > 0$ and any configuration $P \subset \R^d$:
\begin{itemize}
    \item[(WS)] If $d_{Q(0, R)}(A) > \m_{Q(0, R)}(P)$, then $I_P(A) > 0$.
\end{itemize}
For our purposes, however, we will need to strengthen this simple property in two ways:
first to obtain a uniform lower bound on $I_P(A)$ which depends only on $R$ and the slack $d_{Q(0, R)}(A) - \m_{Q(0, R)}(P)$, but not on the specific set $A \subseteq Q(0, R)$;
and then to make the proportion $I_P(A \cap Q(0, R))/R^d$ of copies of $P$ uniform also on the size $R$ of the cube considered.

The first strengthening can be obtained from (WS) by a compactness argument, using the fact that the counting function of admissible configurations is weak$^*$ continuous:

\begin{lem}[Weak supersaturation] \label{lem:weak_supersat}
Let $P \subset \R^d$ be an admissible configuration.
For every $R > 0$ and $\varepsilon > 0$ there exists $c_0 > 0$ such that the following holds:
whenever $A \subseteq Q(0, R)$ satisfies
$d_{Q(0, R)}(A) \geq \m_{Q(0, R)}(P) + \varepsilon$, we have $I_P(A) \geq c_0$.
\end{lem}

\begin{proof}
Suppose for contradiction that the result is false.
Then there exist $\varepsilon > 0$, $R > 0$ and a sequence $(A_i)_{i \geq 1}$ of subsets of $Q(0, R)$, each of density at least $\m_{Q(0, R)}(P) + \varepsilon$, which satisfy $\lim_{i \rightarrow \infty} I_P(A_i) = 0$.

The unit ball $\ball_{\infty}$ of $L^{\infty}(Q(0, R))$ is weak$^*$ compact by the Banach-Alaoglu Theorem, and it is also metrizable in this topology (see \cite[Chapter~2.6]{Megginson98}).
By possibly restricting to a subsequence, we may then assume that $(A_i)_{i \geq 1}$ converges in the weak$^*$ topology of $L^{\infty}(Q(0, R))$;
let us denote its limit by $A \in \ball_{\infty}$.
It is clear that $0 \leq A \leq 1$ almost everywhere, and
$$\frac{1}{R^d} \int_{Q(0, R)} A(x) \,dx = \lim_{i \rightarrow \infty} \frac{1}{R^d} \int_{Q(0, R)} A_i(x) \,dx \geq \m_{Q(0, R)}(P) + \varepsilon.$$
By weak$^*$ continuity of $I_P$ (Lemma \ref{lem:weak*cont_space}), we also have $I_P(A) = \lim_{i \rightarrow \infty} I_P(A_i) = 0$.

Now let $B := \big\{x \in Q(0, R):\, A(x) \geq \varepsilon\big\}$.
Since
$$\varepsilon B(x) \leq A(x) < \varepsilon + B(x) \hspace{3mm} \text{for a.e. } x \in Q(0, R),$$
we conclude that $I_P(B) \leq \varepsilon^{-k} I_P(A) = 0$ and
$$d_{Q(0, R)}(B) > \frac{1}{R^d} \int_{Q(0, R)} A(x) \,d\sigma(x) - \varepsilon \geq \m_{Q(0, R)}(P).$$
But this set $B$ contradicts (WS) (or Lemma~\ref{lem:zero_meas_space}), finishing the proof.
\end{proof}

Our desired supersaturation result now follows from a simple averaging argument:

\begin{thm}[Supersaturation Theorem] \label{thm:supersat_space}
Let $P \subset \R^d$ be an admissible configuration and let $\varepsilon > 0$.
There exist constants $c > 0$ and $R_0 > 0$ such that the following holds for all $R \geq R_0$:
if $A \subseteq Q(0, R)$ satisfies
$$d_{Q(0, R)}(A) \geq \m_{Q(0, R)}(P) + \varepsilon,$$
then $I_P(A) \geq c R^d$.
\end{thm}

\begin{proof}
Take $R_1 > 0$ large enough so that $\m_{Q(0, R_1)}(P) \leq \m_{\R^d}(P) + \varepsilon/4$
(see Lemma~\ref{lem:bounds_space}).
We will show that the conclusion of the theorem holds for $R_0 = 4 d R_1/\varepsilon$ and some constant $c > 0$ to be chosen later.

Suppose $R \geq 4 d R_1/\varepsilon$, and let $A \subseteq Q(0, R)$ be a set of density
$$d_{Q(0, R)}(A) \geq \m_{Q(0, R)}(P) + \varepsilon.$$
Since
$\m_{Q(0, R_1)}(P) \leq \m_{\R^d}(P) + \varepsilon/4 \leq \m_{Q(0, R)}(P) + \varepsilon/4,$
we have that
$$\meas(A) = \meas\big(A \cap Q(0, R)\big) \geq \big(\m_{Q(0, R_1)}(P) + 3\varepsilon/4\big) R^d.$$

Let $K := \lfloor R/R_1 \rfloor$, and note that
$$K^d R_1^d > \bigg( 1 - \frac{R_1}{R} \bigg)^d R^d \geq \bigg( 1 - \frac{d R_1}{R} \bigg) R^d.$$
By our assumption on $R$ we conclude that $K^d R_1^d \geq (1 - \varepsilon/4) R^d$, and thus
\begin{align*}
    \meas\big(A \cap Q(0, K R_1)\big)
    &\geq \meas(A) - \meas\big(Q(0, R) \setminus Q(0, K R_1)\big) \\
    &\geq \big(\m_{Q(0, R_1)}(P) + \varepsilon/2\big) K^d R_1^d.
\end{align*}
Partitioning the cube $Q(0, K R_1)$ into $K^d$ cubes of side length $R_1$, by averaging we conclude that at least $\varepsilon K^d/4$ of these cubes $Q(x, R_1)$ satisfy
\begin{equation} \label{eq:high_dens}
    \meas\big(A \cap Q(x, R_1)\big) \geq \big(\m_{Q(0, R_1)}(P) + \varepsilon/4\big) R_1^d.
\end{equation}
By Lemma~\ref{lem:weak_supersat}, there is some $c_0 > 0$ (depending on $R_1$ and $\varepsilon$ but not on $A$ or $R$) such that $I_P(A \cap Q(x, R_1)) \geq c_0$ holds for each one of the cubes in the partition satisfying~\eqref{eq:high_dens};
summing up all these values we conclude that
$$I_P(A) \geq \frac{\varepsilon K^d}{4} c_0 > \frac{\varepsilon c_0}{4} \bigg(\frac{R}{R_1} - 1\bigg)^d
\geq \frac{\varepsilon c_0}{4} \frac{1}{(2R_1)^d} R^d,$$
finishing the proof for $c = \varepsilon c_0/(2^{d+2} R_1^d)$.
\end{proof}

\begin{remark}
The arguments used in the proofs of Lemma~\ref{lem:weak_supersat} and Theorem~\ref{thm:supersat_space} easily extend to the case of several configurations $P_1, \dots, P_n \subset \R^d$, with only minor and notational modifications.
In the case of the Supersaturation Theorem, one concludes that $I_{P_i}(A) \geq c(\varepsilon) R^d$ holds for some $1 \leq i \leq n$ whenever the density condition
$d_{Q(0, R)}(A) \geq \m_{Q(0, R)}(P_1,\, \dots,\, P_n) + \varepsilon$
is satisfied (assuming $R$ is large enough and all the configurations $P_i$ are admissible).
\end{remark}

Following Bukh~\cite{Bukh}, for each $\delta > 0$ and $\gamma > 0$ we define the \emph{zooming-out operator} $\mathcal{Z}_{\delta}(\gamma)$ as the map which takes a measurable set $A \subseteq \R^d$ to the set
$$\mathcal{Z}_{\delta}(\gamma)[A] := \big\{ x \in \R^d:\, d_{Q(x, \delta)}(A) \geq \gamma \big\}.$$
Intuitively, $\mathcal{Z}_{\delta}(\gamma)[A]$ represents the points where $A$ is not too sparse at scale $\delta$.

Using the Supersaturation Theorem together with the Counting Lemma, we can now show that the existence of copies of $P$ in a set $A$ follows also from the weaker assumption that its zoomed-out version $\mathcal{Z}_{\delta}(\gamma)[A]$ has density higher than $\m_{\R^d}(P)$ (rather than $A$ itself having this same density).
This property will be important for us later on.

\begin{cor} \label{cor:supersat_space}
Given an admissible configuration $P \subset \R^d$ and $\varepsilon > 0$, there exists $\delta_0 > 0$ such that the following holds for all $\delta \leq \delta_0$:
if $A \subseteq \R^d$ satisfies
$$\overline{d} \big(\mathcal{Z}_{\delta}(\varepsilon)[A]\big) \geq \m_{\R^d}(P) + \varepsilon,$$
then $A$ contains a congruent copy of $P$.
\end{cor}

\begin{proof}
Let $R_0$, $c > 0$ be the constants promised in the Supersaturation Theorem applied to $P$ and with $\varepsilon$ substituted by $\varepsilon/3$.
Up to substituting $R_0$ by some larger constant, we may also assume that $\m_{Q(0, R)}(P) \leq \m_{\R^d}(P) + \varepsilon/3$ for all $R \geq R_0$
(see Lemma~\ref{lem:bounds_space}).

Suppose $A \subseteq \R^d$ satisfies
$\overline{d} \big(\mathcal{Z}_{\delta}(\varepsilon)[A]\big) \geq \m_{\R^d}(P) + \varepsilon$ for some $0 < \delta \leq 1$.
Since
\begin{align*}
    \limsup_{R\rightarrow \infty} d_{Q(0,R)} \big(\mathcal{Z}_{\delta}(\varepsilon)[A \cap Q(0,R)]\big)
    &= \limsup_{R\rightarrow \infty} d_{Q(0,R)} \big(\mathcal{Z}_{\delta}(\varepsilon)[A]\big) \\
    &= \overline{d} \big(\mathcal{Z}_{\delta}(\varepsilon)[A]\big)
    \geq \m_{\R^d}(P) + \varepsilon,
\end{align*}
there must exist some $R \geq R_0$ such that
$$d_{Q(0,R)} \big(\mathcal{Z}_{\delta}(\varepsilon)[A \cap Q(0,R)]\big) \geq \m_{\R^d}(P) + 2\varepsilon/3.$$
Denoting $A' := A\cap Q(0,R)$, we may then assume that $A' \subseteq Q(0, R)$ satisfies
\begin{equation} \label{eq:zoom}
    d_{Q(0, R)} \big(\mathcal{Z}_{\delta}(\varepsilon)[A']\big) \geq \m_{Q(0, R)}(P) + \varepsilon/3
\end{equation}
for some $R \geq R_0$, and wish to show that $A'$ (and hence $A$) contains a copy of $P$ if $\delta > 0$ is small enough depending on $P$ and $\varepsilon$.

By the Supersaturation Theorem, inequality~\eqref{eq:zoom} implies that $I_P \big(\mathcal{Z}_{\delta}(\varepsilon)[A']\big) \geq c R^d$.
Since
$A' * \mathcal{Q}_{\delta}(x) \geq \varepsilon \cdot \mathcal{Z}_{\delta}(\varepsilon)[A'](x)$ for all $x \in \R^d$,
we obtain from the Counting Lemma that
\begin{align*}
    I_P(A') \geq I_P(A' * \mathcal{Q}_{\delta}) - C_P \delta^{1/4} R^d
    &\geq \varepsilon^{|P|} I_P \big(\mathcal{Z}_{\delta}(\varepsilon)[A']\big) - C_P \delta^{1/4} R^d \\
    &\geq \big(\varepsilon^{|P|} c - C_P \delta^{1/4}\big) R^d.
\end{align*}
Taking $\delta > 0$ small enough for this last expression to be positive we conclude that $I_P(A') > 0$, and so $A'$ contains a copy of $P$ as wished.
\end{proof}

\subsection{Results on the independence density} \label{sec:several_config}

We are finally in a position to properly study the independence density parameter for a family of configurations in Euclidean space.


We start by proving a simple lower bound on the independence density of several distinct configurations;
this result and the argument we use to prove it are originally due to Bukh~\cite{Bukh}.

\begin{lem}[Supermultiplicativity] \label{lem:lower_bound_space}
For all $n \geq 1$ and all configurations $P_1, \dots, P_n \subset \R^d$, we have that
$$\m_{\R^d}(P_1,\, P_2,\, \dots,\, P_n) \geq \prod_{i=1}^n \m_{\R^d}(P_i).$$
\end{lem}

\begin{proof}
Fix $\varepsilon > 0$ and choose $R$ large enough so that
$$\min_{1 \leq i \leq n} (R - \diam P_i)^d \geq (1 - \varepsilon) R^d.$$
For each $1 \leq i \leq n$, let $A_i \subseteq Q(0, R - \diam P_i)$ be a set which avoids $P_i$ and satisfies $d_{Q(0, R - \diam P_i)}(A_i) > \m_{\R^d}(P_i) - \varepsilon$ (this is possible by Lemma~\ref{lem:bounds_space}).
We then construct the $R$-periodic set
$A_i' := A_i + R \mathbb{Z}^d$, which also avoids $P_i$ and has density
$$d(A_i') = \frac{(R - \diam P_i)^d}{R^d} d_{Q(0, R - \diam P_i)}(A_i) > \m_{\R^d}(P_i) - 2\varepsilon.$$

Since each set $A_i'$ is periodic with the same fundamental domain $Q(0, R)$, it follows that the average of $d\big(\bigcap_{i=1}^n (x_i + A_i')\big)$ over independent translates $x_1, \dots, x_n \in Q(0, R)$ is equal to $\prod_{i=1}^n d(A_i')$.
There must then exist some $x_1, \dots, x_n \in Q(0, R)$ for which $$d\bigg(\bigcap_{i=1}^n (x_i + A_i')\bigg) \geq \prod_{i=1}^n d(A_i') > \prod_{i=1}^n (\m_{\R^d}(P_i) - 2\varepsilon).$$
Since $\bigcap_{i=1}^n (x_i + A_i')$ avoids each of the configurations $P_i$ and $\varepsilon > 0$ was arbitrary, the desired lower bound follows.
\end{proof}

Intuitively, one may regard $\m_{\R^d}(P_1,\, P_2,\, \dots,\, P_n)$ being close to $\prod_{i=1}^n \m_{\R^d}(P_i)$ as some sort of \emph{independence} or \emph{lack of correlation} between the $n$ constraints of forbidding each configuration $P_i$;
in this case, there is no better way to choose a set avoiding all of these configurations than simply intersecting optimal $P_i$-avoiding sets for each $i$ (after suitably translating them).
One might then expect this to happen if the sizes of each $P_i$ are very different from each other, so that each constraint will be relevant in different and largely independent scales.

Our next result shows this is indeed the case whenever the configurations considered are all admissible.
(A theorem of Graham~\cite{Graham} implies this is \emph{not} necessarily true if the configurations considered are non-admissible;
see Section~\ref{Conclusion} for a discussion.)
The proof we present here is based on Bukh's arguments for supersaturable properties, and generalizes his result from two-point configurations to general admissible configurations.

\begin{thm}[Asymptotic independence] \label{thm:main_space}
If $P_1, P_2, \dots, P_n \subset \R^d$ are admissible configurations, then
$$\m_{\R^d}(t_1 P_1,\, t_2 P_2,\, \dots,\, t_n P_n) \rightarrow \prod_{i=1}^n \m_{\R^d}(P_i)$$
as the ratios $t_2/t_1,\, t_3/t_2,\, \dots,\, t_n/t_{n-1}$ tend to infinity.
\end{thm}

\begin{proof}
We have already seen that
$$\m_{\R^d}(t_1 P_1,\, t_2 P_2,\, \dots,\, t_n P_n) \geq \prod_{i=1}^n \m_{\R^d}(t_i P_i) = \prod_{i=1}^n \m_{\R^d}(P_i)$$
always holds, so it suffices to show that $\m_{\R^d}(t_1 P_1,\, t_2 P_2,\, \dots,\, t_n P_n)$ is no larger than $\prod_{i=1}^n \m_{\R^d}(P_i) + \varepsilon$ whenever $\varepsilon > 0$ and the ratios between consecutive scales $t_i$ are large enough.
We shall proceed by induction, with the case $n = 1$ being trivial.

Let $n \geq 2$ and suppose the theorem holds for configurations $P_1, \dots, P_{n-1}$.
Fix $0 < \varepsilon \leq 1$, and let $t_1,\, \dots,\, t_{n-1} > 0$ be dilation parameters for which
$$\m_{\R^d}(t_1 P_1,\, \dots,\, t_{n-1} P_{n-1}) \leq \prod_{i=1}^{n-1} \m_{\R^d}(P_i) + \varepsilon;$$
now take $R_0 > 0$ large enough so that
$$\m_{Q(0, R)}(t_1 P_1,\, \dots,\, t_{n-1} P_{n-1}) \leq \m_{\R^d}(t_1 P_1,\, \dots,\, t_{n-1} P_{n-1}) + \varepsilon$$
holds for all $R \geq R_0$
(this quantity exists by Lemma~\ref{lem:bounds_space}).

If $A \subseteq \R^d$ is a measurable set avoiding $t_1 P_1,\, \dots,\, t_{n-1} P_{n-1}$, then clearly
\begin{equation} \label{eq:1}
    d_{Q(x, R)}(A) \leq \m_{Q(0, R)}(t_1 P_1,\, \dots,\, t_{n-1} P_{n-1}) \hspace{5mm} \text{for all } x \in \R^d,\, R > 0.
\end{equation}
Moreover, if $A$ also avoids $t_n P_n$ for some $t_n > 0$, then $A/t_n$ avoids $P_n$ and so by Corollary~\ref{cor:supersat_space} there is some $\delta_0 > 0$ (depending only on $P_n$ and $\varepsilon$) for which
\begin{equation} \label{eq:2}
    \overline{d}\big(\mathcal{Z}_{\delta}(\varepsilon)[A/t_n]\big) \leq \m_{\R^d}(P_n) + \varepsilon \hspace{5mm} \text{for all } \delta \leq \delta_0.
\end{equation}

Suppose now that $t_n \geq R_0/\delta_0$, and let $A \subseteq \R^d$ be any measurable set avoiding $t_1 P_1,\, \dots,\, t_n P_n$.
We conclude from (\ref{eq:1}) that
\begin{align*}
    d_{Q(x,\, t_n \delta_0)}(A) &\leq \m_{Q(0,\, t_n \delta_0)}(t_1 P_1,\, \dots,\, t_{n-1} P_{n-1}) \\
    &\leq \m_{\R^d}(t_1 P_1,\, \dots,\, t_{n-1} P_{n-1}) + \varepsilon \\
    &\leq \prod_{i=1}^{n-1} \m_{\R^d}(P_i) + 2\varepsilon
\end{align*}
holds for all $x \in \R^d$, and from (\ref{eq:2}) we have
$$\overline{d}\big(\mathcal{Z}_{t_n \delta_0}(\varepsilon)[A]\big) = \overline{d}\big(\mathcal{Z}_{\delta_0}(\varepsilon)[A/t_n]\big) \leq \m_{\R^d}(P_n) + \varepsilon.$$
This means that the density of $A$ inside cubes $Q(x, t_n \delta_0)$ of side length $t_n \delta_0$ is at most $\varepsilon$
(when $x \notin \mathcal{Z}_{t_n \delta_0}(\varepsilon)[A]$)
except at a set of upper density at most $\m_{\R^d}(P_n) + \varepsilon$, when it is instead no more than $\prod_{i=1}^{n-1} \m_{\R^d}(P_i) + 2\varepsilon$.
Taking averages, we conclude that
$$\overline{d}(A) \leq \varepsilon + \big( \m_{\R^d}(P_n) + \varepsilon \big) \bigg( \prod_{i=1}^{n-1} \m_{\R^d}(P_i) + 2\varepsilon \bigg)
\leq \prod_{i=1}^n \m_{\R^d}(P_i) + 6\varepsilon$$
(where we used that $0 < \varepsilon \leq 1$).
This inequality finishes the proof.
\end{proof}

As an immediate corollary of the last theorem, we conclude that
$$\m_{\R^d}(t_1 P,\, t_2 P,\, \dots,\, t_n P) \rightarrow \m_{\R^d}(P)^n$$
as $t_2/t_1,\, t_3/t_2,\, \dots,\, t_n/t_{n-1} \rightarrow \infty$ whenever $P \subset \R^d$ is admissible;
this answers our question (Q1) in the case of admissible configurations.
Let us now show how this result easily implies Bourgain's Theorem given in the Introduction:

\begin{proof}[Proof of Theorem~\ref{thmbourgain}]
Suppose $A \subset \R^d$ is a measurable set not satisfying the conclusion of the theorem;
thus there is a sequence $(t_j)_{j \geq 1}$ tending to infinity such that $A$ does not contain a copy of any $t_j P$.
This implies that $\overline{d}(A) \leq \m_{\R^d}(t_1 P,\, t_2 P, \dots,\, t_n P)$ for all $n \in \mathbb{N}$.
By taking a suitably fast-growing subsequence, we may then use Theorem~\ref{thm:main_space} to obtain (say) $\overline{d}(A) \leq 2 \m_{\R^d}(P)^n$ for any fixed $n \geq 1$.
Since $\m_{\R^d}(P) < 1$,\footnote{An easy averaging argument shows that $\m_{\R^d}(P) \leq 1 - 1/|P|$;
see Lemma~\ref{lem:reassuring} for a proof of this inequality in the spherical setting.}
this implies that $\overline{d}(A) = 0$, as wished.
\end{proof}

Going back to our study of the independence density for multiple configurations, we will now consider the opposite situation of what we have seen before:
when the constraints of forbidding each individual configuration are so strongly correlated as to be essentially redundant.
One might expect this to be the case, for instance, when we are forbidding very close dilates of a given configuration $P$.

We will show that this intuition is indeed correct, whether or not the configuration considered is admissible, and the proof is much simpler than in the case of very distant dilates of $P$ (in particular not needing the results from earlier sections).

\begin{lem}[Asymptotic redundancy] \label{lem:tend_to_one}
For any configuration $P \subset \R^d$, we have that
$$\m_{\R^d}(t_1 P,\, t_2 P,\, \dots,\, t_n P) \rightarrow \m_{\R^d}(P)$$
as $t_2/t_1,\, t_3/t_2,\, \dots,\, t_n/t_{n-1} \rightarrow 1$.
\end{lem}

\begin{proof}
Assume by dilation invariance that $t_1 = 1$, and note that it suffices to show the convergence above with $\m_{\R^d}$ replaced by $\m_{Q(0, R)}$ for every fixed $R > 0$.
We will then fix an arbitrary $R > 0$ and prove that $\m_{Q(0, R)}(P,\, t_2 P,\, \dots,\, t_n P) \rightarrow \m_{Q(0, R)}(P)$ as $t_2,\, t_3,\, \dots,\, t_n \rightarrow 1$.

Let $(v_1, v_2, \dots, v_k)$ be an ordering of the points of $P$, and consider the continuous function $g_P:\, (\R^d)^k \times \SO(\R^d) \rightarrow \R$ given by
$$g_P(x_1,\, \dots,\, x_k,\, T) := \sum_{j=2}^k \| (x_j - x_1) - T (v_j - v_1) \|.$$
Note that $\min_{T \in \SO(\R^d)} g_P(x_1,\, \dots,\, x_k,\, T) = 0$ if and only if $(x_1, \dots, x_k)$ is congruent to $(v_1, \dots, v_k)$.

Fix some $\varepsilon > 0$, and let $A \subset Q(0, R)$ be a measurable set which avoids $P$ and has density $d_{Q(0, R)}(A) \geq \m_{Q(0, R)}(P) - \varepsilon$.
By inner regularity, we know there exists a compact set $\widetilde{A} \subseteq A$ with $d_{Q(0, R)}(\widetilde{A}) \geq \m_{Q(0, R)}(P) - 2\varepsilon$.
Denote by $\gamma$ the minimum of the continuous function $g_P$ on the compact set $\widetilde{A}^k \times \SO(\R^d)$;
since $\widetilde{A}$ avoids $P$, it follows that $\gamma > 0$.

We will now prove that $\widetilde{A}$ also avoids $t P$ whenever $t$ is sufficiently close to 1, say when $|t-1| < \gamma/(k \cdot \diam P)$.
Indeed, for all $x_1, \dots, x_k \in \widetilde{A}$ and all $T \in \SO(\R^d)$, by the triangle inequality we have that
\begin{align*}
    \sum_{j=2}^k \| (x_j - x_1) - T (t v_j - t v_1) \|
    &\geq \sum_{j=2}^k \big| \|(x_j - x_1) - T (v_j - v_1)\| - |t-1| \|v_j - v_1\| \big| \\
    &> \sum_{j=2}^k \| (x_j - x_1) - T (v_j - v_1) \| - k \cdot |t-1| \,\diam P \\
    &\geq \gamma - k \cdot |t-1| \,\diam P,
\end{align*}
which is positive if $|t-1| < \gamma/(k \cdot \diam P)$.
In particular, we see that
$$\m_{Q(0, R)}(P,\, t_2 P,\, \dots,\, t_n P) \geq d_{Q(0, R)}(\widetilde{A}) \geq \m_{Q(0, R)}(P) - 2\varepsilon$$
whenever $|t_j - 1| < \gamma/(k \cdot \diam P)$ for $2 \leq j \leq n$.
Since we clearly have that $\m_{Q(0, R)}(P,\, t_2 P,\, \dots,\, t_n P) \leq \m_{Q(0, R)}(P)$, the result follows.
\end{proof}

The proof of this last result actually implies a somewhat stronger and more technical property of the independence density, namely that every configuration $P$ where $\m_{\R^d}$ is \emph{discontinuous} must be a local minimum across the `discontinuity barrier';
more formally, we have that
$\inf_{P'\in \ball(P, \delta)} \m_{\R^d}(P') \rightarrow \m_{\R^d}(P)$
as $\delta \rightarrow 0$, where $\ball(P, \delta)$ is the ball of radius $\delta$ centered on $P$.
(The details of the proof are given below.)
If the configuration $P$ is admissible, then we can also prove the corresponding limit for $\sup_{P'\in \ball(P, \delta)} \m_{\R^d}(P')$ and conclude that $\m_{\R^d}$ is in fact continuous at this point.
This is done in the next theorem:

\begin{thm}[Continuity of the independence density] \label{thm:cont_of_m_space}
For every $n \geq 1$, the function
$(P_1, \dots, P_n) \mapsto \m_{\R^d}(P_1, \dots, P_n)$
is continuous on the set of $n$ admissible configurations in $\R^d$.
\end{thm}

\begin{proof}
For the sake of better readability, we will prove the result in the case of only one forbidden configuration;
the $n$-variable version easily follows from the same argument.
Fix some $\varepsilon > 0$, and let $R_1 \geq 1$ be large enough so that
$$\m_{\R^d}(P') \leq \m_{Q(0, R)}(P') \leq \m_{\R^d}(P') + \varepsilon$$
holds for all $P' \in \ball(P, 1)$ and all $R \geq R_1$
(this value exists by Lemma~\ref{lem:bounds_space}).

Let $R \geq R_1$ and let $A \subset Q(0, R)$ be a compact $P$-avoiding set with density
$$d_{Q(0, R)}(A) \geq \m_{Q(0, R)}(P) - \varepsilon.$$
Proceeding exactly as we did in the proof of the last lemma, we conclude that $A$ also avoids all $P'$ close enough to $P$;
for all such configurations we then have
$$\m_{Q(0, R)}(P') \geq d_{Q(0, R)}(A) \geq \m_{Q(0, R)}(P) - \varepsilon \geq \m_{\R^d}(P) - \varepsilon.$$
Since $\m_{Q(0, R)}(P') \leq \m_{\R^d}(P') + \varepsilon$ whenever $P' \in \ball(P, 1)$, this implies that $\m_{\R^d}(P') \geq \m_{\R^d}(P) - 2\varepsilon$ for all $P'$ close enough to $P$.

Now we suppose that $P$ is admissible, and let $R_0$, $c > 0$ be the constants promised by the Supersaturation Theorem
(Theorem \ref{thm:supersat_space}).
Let $R \geq \max\{R_0, R_1\}$.
By equicontinuity (Lemma~\ref{lem:equicont_space}), there is some $\delta > 0$ for which the inequality
\begin{equation} \label{eq:equicont}
    |I_{P'}(A) - I_P(A)| < c R^d \hspace{5mm} \text{for all } P' \in \ball(P, \delta)
\end{equation}
holds whenever $A \subset Q(0, R)$ is a measurable set;
fix such a value of $\delta$.

If $P' \in \ball(P, \delta)$ and $A \subset Q(0, R)$ is a measurable set avoiding $P'$, we conclude from inequality~\eqref{eq:equicont} that $I_P(A) < c R^d$.
By the Supersaturation Theorem this implies that $d_{Q(0, R)}(A) < \m_{Q(0, R)}(P) + \varepsilon$, and thus (by optimizing over $A$) we conclude that $\m_{Q(0, R)}(P') \leq \m_{Q(0, R)}(P) + \varepsilon$.
It follows that
$$\m_{\R^d}(P') \leq \m_{Q(0, R)}(P') \leq \m_{Q(0, R)}(P) + \varepsilon \leq \m_{\R^d}(P) + 2\varepsilon$$
whenever $P' \in \ball(P, \delta)$, finishing the proof.
\end{proof}

These last results can now be combined in a very simple way to give an (almost complete) answer to question (Q2), when restricted to admissible configurations.
Let us denote by $\mathcal{M}_n(P)$ the set of all possible independence densities one can obtain by forbidding $n$ distinct dilates of a configuration $P$, that is
$$\mathcal{M}_n(P) := \big\{\m_{\R^d}(t_1 P,\, t_2 P,\, \dots,\, t_n P):\, 0 < t_1 < t_2 < \dots < t_n < \infty\big\}.$$
Recall that (Q2) asked for an explicit description of this set $\mathcal{M}_n(P)$.

\begin{thm}[Forbidding multiple dilates] \label{thm:description_of_M}
If $P \subset \R^d$ is admissible, then
$$\big(\m_{\R^d}(P)^n,\, \m_{\R^d}(P)\big) \subseteq \mathcal{M}_n(P) \subseteq \big[\m_{\R^d}(P)^n,\, \m_{\R^d}(P)\big].$$
\end{thm}

\begin{proof}
It is clear that $\m_{\R^d}(t_1 P,\, t_2 P,\, \dots,\, t_n P) \leq \m_{\R^d}(t_1 P) = \m_{\R^d}(P)$ always holds, and we saw in Lemma~\ref{lem:lower_bound_space} that
$$\m_{\R^d}(t_1 P,\, t_2 P,\, \dots,\, t_n P) \geq \prod_{i=1}^n \m_{\R^d}(t_i P) = \m_{\R^d}(P)^n.$$
Moreover, Lemma~\ref{lem:tend_to_one} implies that $\m_{\R^d}(P)$ is an accumulation point of the set $\mathcal{M}_n(P)$, and (since $P$ is admissible) Theorem~\ref{thm:main_space} implies the same about $\m_{\R^d}(P)^n$.
The result follows from continuity of the function
$$(t_1,\, t_2,\, \dots,\, t_n) \mapsto \m_{\R^d}(t_1 P,\, t_2 P,\, \dots,\, t_n P),$$
which is an immediate consequence of Theorem~\ref{thm:cont_of_m_space}.
\end{proof}

As our final result in the Euclidean setting, we will show the existence of extremizer sets which avoid admissible configurations.
This generalizes a result of Bukh (see Corollary 13 in~\cite{Bukh}) from forbidden distances to higher-order configurations.

\begin{thm}[Existence of extremizers]
If $P \subset \R^{d}$ is admissible, then there exists a $P$-avoiding measurable set $A \subseteq \R^d$ with density $d(A) = \m_{\R^d}(P)$.
\end{thm}

\begin{proof}
For each integer $i \geq 1$, let $A_i \subseteq Q(0, i)$ be a $P$-avoiding set with density $d_{Q(0, i)}(A_i) \geq \m_{Q(0, i)}(P) - 2^{-i}$.
Denote the unit ball of $L^{\infty}(\R^d)$ by $\ball_{\infty}$;
by the Banach-Alaoglu Theorem, $\ball_{\infty}$ is weak$^*$ compact.
By restricting to a subsequence if necessary, we may then assume that $(A_i)_{i \geq 1}$ converges to some element $\widetilde{A} \in \ball_{\infty}$ in the weak$^*$ topology of $L^{\infty}(\R^d)$.
Denote by $A := \supp \widetilde{A}$ the support of $\widetilde{A}$.\footnote{Strictly speaking, $\widetilde{A}$ is an equivalence class of functions, not a specific function.
More formally, our set $A$ is the support of an (arbitrary) representative of this class, but since this choice of representative makes no difference to our argument one can ignore this technicality.}

We will first prove that $I_P(A) = 0$.
Fix some $R > 0$ and (for notational convenience) denote the indicator function of the cube $Q(0, R)$ by $\chi_R$.
Writing $\chi_R A_i$ for the pointwise product of $\chi_R$ and the indicator function of $A_i$, one easily sees from the definition that $(\chi_R A_i)_{i \geq 1}$ converges to $\chi_R \widetilde{A}$ in the weak$^*$ topology of $L^{\infty}(Q(0, R))$.
As $P$ is admissible, the counting function $I_P$ is weak$^*$ continuous (by Lemma~\ref{lem:weak*cont_space}) and thus $I_P(\chi_R \widetilde{A}) = \lim_{i \rightarrow \infty} I_P(\chi_R A_i) = 0$.
We now proceed as in the proof of Lemma~\ref{lem:weak_supersat} to show that $I_P(\supp \chi_R \widetilde{A}) = 0$ as well:
first approximate $I_P(\supp \chi_R \widetilde{A})$ by $I_P(B_{\varepsilon})$, where $B_{\varepsilon} := \{x \in Q(0, R):\, \widetilde{A}(x) \geq \varepsilon\}$ and $\varepsilon > 0$ is a sufficiently small constant (depending on $R$), and then note that $I_P(B_{\varepsilon}) \leq \varepsilon^{-|P|} I_P(\chi_R \widetilde{A}) = 0$ for all $\varepsilon > 0$.
Since $\supp \chi_R \widetilde{A} = A \cap Q(0, R)$ up to zero-measure sets and $R > 0$ is arbitrary, we conclude that $I_P(A) = 0$ as wished.

Next we prove that $d(A) = \m_{\R^d}(P)$.
Since $I_P(A) = 0$, it follows from Lemma~\ref{lem:zero_meas_space} that $\overline{d}(A) \leq \m_{\R^d}(P)$, and so it suffices to show that
\begin{equation} \label{eq:liminf}
    \liminf_{R \rightarrow \infty} d_{Q(0, R)}(A) \geq \m_{\R^d}(P).
\end{equation}
Fix some arbitrary $\varepsilon > 0$ and take $R_0 \geq 2$ large enough so that $(R_0 + 2\diam P)^d < (1 + \varepsilon/4) R_0^d$.
For any given $R \geq R_0$, take a $P$-avoiding set $B_R \subseteq Q(0, R)$ with
$$d_{Q(0, R)}(B_R) > \m_{Q(0, R)}(P) - \varepsilon/4 \geq \m_{\R^d}(P) - \varepsilon/4.$$
For all $i \geq R$, define $A_i' := B_R \cup (A_i \setminus Q(0, R + 2\diam P))$;
note that $A_i'$ avoids $P$ and
\begin{align*}
    \meas(A_i') &= \meas(A_i) - \meas\big(A_i \cap \big( Q(0, R + 2\diam P) \setminus Q(0, R) \big)\big) \\
    &\hspace{1cm} - \meas(A_i \cap Q(0, R)) + \meas(B_R) \\
    &\geq \meas(A_i) - \big((R + 2\diam P)^d - R^d\big) + \meas(B_R) - \meas(A \cap Q(0, R)) \\
    &\hspace{1cm} + \meas(A \cap Q(0, R)) - \meas(A_i \cap Q(0, R)) \\
    &\geq \big(\m_{Q(0, i)}(P) - 2^{-i}\big) i^d - \frac{\varepsilon R^d}{4} + \big(d_{Q(0, R)}(B_R) - d_{Q(0, R)}(A)\big) R^d \\
    &\hspace{1cm} + \int_{Q(0, R)} \big(A(x) - A_i(x)\big) \,dx \\
    &\geq \big(\m_{Q(0, i)}(P) - 2^{-i}\big) i^d - \frac{\varepsilon R^d}{2} + \big(\m_{\R^d}(P) - d_{Q(0, R)}(A)\big) R^d \\
    &\hspace{1cm} + \int_{Q(0, R)} \big(\widetilde{A}(x) - A_i(x)\big) \,dx.
\end{align*}
Since $\meas(A_i') \leq \m_{Q(0, i)}(P)\, i^d$ for all $i \geq R$ and $\int_{Q(0, R)} \big(\widetilde{A}(x) - A_i(x)\big) \,dx > -\varepsilon$ for all sufficiently large $i$, we conclude that for large enough $i$ we have
$$d_{Q(0, R)}(A) > \m_{\R^d}(P) - \frac{i^d}{2^i R^d} - \frac{\varepsilon}{2} - \frac{\varepsilon}{R^d} > \m_{\R^d}(P) - \varepsilon,$$
proving inequality \eqref{eq:liminf}.

Finally, since $I_P(A) = 0$, it follows from Lemma~\ref{lem:zero_meas_space} that we can remove a zero-measure subset of $A$ in order to remove all copies of $P$ without changing its density.
The theorem follows.
\end{proof}

\section{Configurations on the sphere} \label{Sphere}

In this section we turn to the question of whether the methods and results shown in the Euclidean space setting can also be made to work in the spherical setting.

We shall fix an integer $d \geq 2$ throughout this section and work on the $d$-dimensional unit sphere $\S^d := \big\{x\in \R^{d+1}:\, \|x\|=1\big\}$.
We denote the uniform probability measure on $\S^d$ by $\sigma^{(d)} = \sigma$, and the normalized Haar measure on $\SO(\R^{d+1})$ by $\mu_{d+1} = \mu$.
These two measures are related as follows:
if $X \subseteq \S^d$ is a measurable set and $x\in \S^d$, then
$$\sigma(X) = \mu \big(\big\{T\in \SO(\R^{d+1}):\, Tx \in X\big\}\big).$$

The analogue of the axis-parallel cube in the spherical setting will be the \emph{spherical cap}:
given $x \in \S^d$ and $\rho > 0$, we
denote\footnote{It is more customary to define the spherical cap using angular distance instead of Euclidean distance as we use.
There is no meaningful (qualitative) difference between these two choices, but the use of the Euclidean distance will be more convenient for us.}
$$\Scap(x, \rho) := \big\{y \in \S^d:\, \|x - y\|_{\R^{d+1}} \leq \rho\big\}.$$
We say $\Scap(x, \rho)$ is the \emph{spherical cap with center $x$ and radius $\rho$}.
Since its measure $\sigma(\Scap(x, \rho))$ does not depend on the center point $x$, we shall denote this value simply by $\sigma(\Scap_{\rho})$.
For a given (measurable) set $A \subseteq \S^d$ we then write
$$d_{\Scap(x, \rho)}(A) := \frac{\sigma(A \cap \Scap(x, \rho))}{\sigma(\Scap_{\rho})}$$
for the density of $A$ inside this cap.

We define a (spherical) configuration on $\S^d$ as a finite subset of $\R^{d+1}$ which is congruent to a set on $\S^d$;
it is convenient to allow for configurations that are not necessarily on the sphere in order to consider dilations.
Note that, if $P, Q \subset \S^d$ are two configurations which \emph{are} on the sphere, then $P \simeq Q$ if and only if there is a transformation $T \in \SO(\R^{d+1})$ for which $P = T \cdot Q$ (translations are no longer necessary in this case).

A spherical configuration $P$ on $\S^d$ is said to be \emph{admissible} if it has at most $d$ points and if it is congruent to a set $P' \subset \S^d$ which is
linearly independent.\footnote{Note that this definition is different from the one in the Euclidean setting, where we required the points to be affinely independent instead of linearly independent.
The reason behind this difference is that the Euclidean space is translation-invariant while the sphere is not, so affine properties on $\R^d$ should translate to linear properties on $\S^d$.}
As before, we shall say that some set $A \subseteq \S^d$ \emph{avoids} $P$ if there is no subset of $A$ which is congruent to $P$.

The natural analogues of the independence density in the spherical setting can now be given.
For $n \geq 1$ configurations $P_1, \dots, P_n$ on $\S^d$, we define the quantities
\begin{align*}
    \m_{\S^d}(P_1, \dots, P_n) &:= \sup \big\{\sigma(A):\, A \subset \S^d \text{ avoids } P_i,\, 1 \leq i \leq n\big\} \hspace{2mm} \text{and} \\
    \m_{\Scap(x, \rho)}(P_1, \dots, P_n) &:= \sup \big\{d_{\Scap(x, \rho)}(A):\, A \subset \Scap(x, \rho) \text{ avoids } P_i,\, 1 \leq i \leq n\big\}.
\end{align*}
Whenever convenient we will state and prove results in the case of only one forbidden configuration, as the more general case of multiple forbidden configurations follows from the same arguments with only trivial modifications (but heavier notation).

\medskip

The first issue we encounter in the spherical setting is that it is not compatible with dilations:
given a set of points $P \subset \S^d$ and some dilation parameter $t > 0$, it is usually \emph{not} true that there exists a set $Q \subset \S^d$ congruent to $t P$.
However, there is a large class of configurations (including the ones we call admissible) for which this \emph{is} true whenever $0 < t \leq 1$;
we shall say that they are \emph{contractible}.

It is easy to show that any configuration $P \subset \S^d$ which is contained in a $d$-dimensional affine hyperplane
(e.g. any configuration with at most $d+1$ points)
is contractible.
Indeed, let $0 < t\leq 1$ and suppose $P \subset \S^d \cap (w+U)$, where $U \subset \R^{d+1}$ is a $d$-dimensional subspace and $w$ is orthogonal to $U$.
Then $w$ is orthogonal to $v-w$ for every $v\in P$, and one readily checks
that\footnote{This is true if $w\neq 0$, by two applications of Pythagoras' Theorem.
If $w=0$, then one has instead that $(1-t^2)^{1/2} u + tP \subset \S^d$ for a unit vector $u\in \R^{d+1}$ orthogonal to the subspace $U$.}
$sw + tP \subset \S^d$ for $s = \big(t^2 + (1-t^2) \|w\|^{-2}\big)^{1/2} - t$.

Even when the configuration we are considering is contractible, however, there is no easy relationship between the independence densities of its distinct dilates.
We will then start with the following reassuring lemma, which in a sense assures us the results we will eventually obtain are not true for only trivial reasons.

\begin{lem} \label{lem:reassuring}
For any fixed contractible configuration $P \subset \S^d$, we have that
$$\inf_{0 < t \leq 1} \m_{\S^d}(tP) > 0 \hspace{3mm} \text{and} \hspace{3mm} \sup_{0 < t \leq 1} \m_{\S^d}(tP) < 1.$$
\end{lem}

\begin{proof}
For the first inequality, we note that spherical caps are exactly the closed balls of the separable metric space $\S^d$ endowed with the Euclidean distance.
This allows us to use the Vitali Covering Lemma;
see for instance \cite[Theorem~2.1]{Mattila95}.
For any given $0 < t \leq 1$, start with the trivial cover $\S^d = \bigcup_{x\in \S^d} \Scap(x,\, \diam tP)$ and apply the Vitali Covering Lemma to obtain a (necessarily finite) set of center points $\{x_1, \dots, x_N\} \subset \S^d$ such that $\Scap(x_i, \diam tP) \cap \Scap(x_j, \diam tP) = \emptyset$ for $i \neq j$ and
$$\S^d = \bigcup_{i=1}^N \Scap(x_i,\, 5 \diam tP).$$
Since the caps $\Scap(x_i, \diam tP)$ are pairwise disjoint, it is easy to see that the set
$$A_t := \bigcup_{i=1}^N \Scap \bigg(x_i,\, \frac{\diam tP}{4}\bigg)$$
does not contain any copy of $tP$.
Finally, as the inequality $\sigma(\Scap_{\rho}) \geq c_d \sigma(\Scap_{20\rho})$ holds for some constant $c_d > 0$ and all $0 < \rho \leq 2$, denoting $\rho(t) := (\diam tP)/4$ we have that
\begin{align*}
    \sigma(A_t)
    = \sum_{i=1}^N \sigma(\Scap_{\rho(t)})
    \geq c_d \sum_{i=1}^N \sigma(\Scap_{20\rho(t)})
    \geq c_d \,\sigma \bigg(\bigcup_{i=1}^N \Scap\big(x_i,\, 20\rho(t)\big)\bigg) = c_d,
\end{align*}
and thus $\m_{\S^d}(tP) \geq \sigma(A_t) \geq c_d$ for all $0 < t \leq 1$.

For the second inequality, suppose $A \subseteq \S^d$ avoids $P = \{v_1, \dots, v_k\}$.
Then
$$\sum_{i=1}^k A(R v_i) = |A \cap R P| \leq k-1 \hspace{5mm} \text{for all } R \in \SO(\R^{d+1}).$$
Integrating over $\SO(\R^{d+1})$, we obtain
$$k \sigma(A) = \int_{\SO(\R^{d+1})} \bigg(\sum_{i=1}^k A(R v_i)\bigg) \,d\mu(R) \leq k-1,$$
implying that $\sigma(A) \leq 1 - 1/k$.
Thus $\sup_{0 < t \leq 1} \m_{\S^d}(tP) \leq 1 - 1/|P|$.
\end{proof}

Given some configuration $P = \{v_1, v_2, \dots, v_k\} \subset \S^d$, we define the \emph{counting function} $I_P$ which acts on a bounded measurable function $f: \S^d \rightarrow \R$ by
$$I_P(f) := \int_{\SO(\R^{d+1})} f(R v_1) f(R v_2) \cdots f(R v_{k}) \,d\mu(R).$$
In the case where $f$ is the indicator function of a set $A \subseteq \S^d$, we note that
$$I_P(A) = \mathbb{P}_{R \in \SO(\R^{d+1})} \big(R v_1,\, R v_2,\, \dots,\, R v_k \in A\big).$$
If the spherical configuration $P$ is \emph{not} a subset of the sphere, we define the function $I_P$ as being equal to $I_Q$ for any $Q \simeq P$ which is contained in $\S^d$.

As in the Euclidean setting, one can show there is no meaningful difference between requiring that a measurable set $A \subseteq \S^d$ avoids some configuration $P$ or that it only satisfies $I_P(A) = 0$.
This is proven in the next lemma:

\begin{lem}[Zero-measure removal] \label{lem:zero_meas_sphere}
Suppose $P \subset \S^d$ is a finite configuration and $A \subseteq \S^d$ is measurable.
If $I_P(A) = 0$, then we can remove a zero-measure subset of $A$ in order to remove all copies of $P$.
\end{lem}

\begin{proof}
It will be more convenient to change spaces and work on the orthogonal group $\SO(\R^{d+1})$ rather than on the sphere $\S^d$.
For $\delta > 0$ and $R \in \SO(\R^{d+1})$, denote by
$$\ball(R, \delta) := \big\{ T \in \SO(\R^{d+1}):\, \|T - R\| \leq \delta \big\}$$
the ball of radius $\delta$ in spectral norm centered on $R$, and let $I$ denote the identity transformation.
We will first show that
\begin{equation} \label{eq:Lebesgue_density_sphere}
    \lim_{\delta \rightarrow 0} \bigg| \frac{1}{\mu(\ball(I, \delta))} \int_{\ball(I, \delta)} A(Tx) \,d\mu(T) - A(x) \bigg| = 0 \text{ for almost every } x \in \S^d.
\end{equation}

Let $e\in \S^d$ be an arbitrary point and define on $\SO(\R^{d+1})$ the (measurable) set $E := \{R \in \SO(\R^{d+1}):\, Re \in A\}$.
By the Lebesgue Density Theorem on $\SO(\R^{d+1})$, we have that
$$\lim_{\delta \rightarrow 0} \bigg| \frac{1}{\mu(\ball(R, \delta))} \int_{\ball(R, \delta)} E(T) \,d\mu(T) - E(R) \bigg| = 0 \text{ for $\mu$-a.e. } R \in \SO(\R^{d+1}).$$
But this means exactly that the measure of the set
$$F := \bigg\{ R \in \SO(\R^{d+1}):\, \lim_{\delta \rightarrow 0} \bigg| \frac{1}{\mu(\ball(I, \delta))} \int_{\ball(I, \delta)} A(T Re) \,d\mu(T) - A(Re)\bigg| \neq 0 \bigg\}$$
of non-density points is zero.
It is clear from the definition of $F$ that it is invariant under the right-action of $\Stab^{\SO(\R^{d+1})}(e)$;
this implies $\sigma(\{Re:\, R \in F\}) = \mu(F) = 0$, proving (\ref{eq:Lebesgue_density_sphere}).

Now we remove from $A$ all points $x$ for which identity (\ref{eq:Lebesgue_density_sphere}) does not hold, thus obtaining a subset $B \subseteq A$ with $\sigma(A \setminus B) = 0$ and
$$\lim_{\delta \rightarrow 0} \frac{1}{\mu(\ball(I, \delta))} \int_{\ball(I, \delta)} B(Tx) \,d\mu(T) = 1 \hspace{5mm} \text{for all } x \in B.$$
We will show that no copy of $P$ remains on this restricted set $B$, which will finish the proof of the lemma.

Suppose for contradiction that $B$ contains a copy $\{u_1, \dots, u_k\}$ of $P$.
Then there exists $\delta > 0$ for which
$$\frac{1}{\mu(\ball(I, \delta))} \int_{\ball(I, \delta)} B(T u_i) \,d\mu(T) \geq 1 - \frac{1}{2k} \hspace{5mm} \text{for all } 1 \leq i \leq k,$$
which means that $\mathbb{P}_{T \in \ball(I, \delta)}(T u_i \notin B) \leq 1/2k$ for each $1 \leq i \leq k$.
Thus
\begin{align*}
    I_P(B) &= \mathbb{P}_{T \in \SO(\R^{d+1})} \big(T u_1,\, \dots,\, T u_k \in B\big) \\
    &\geq \mu(\ball(I, \delta)) \cdot \mathbb{P}_{T \in \ball(I, \delta)} \big(T u_1,\, \dots,\, T u_k \in B\big) \\
    &\geq \mu(\ball(I, \delta)) \bigg( 1 - \sum_{i=1}^k \mathbb{P}_{T \in \ball(I, \delta)}(T u_i \notin B) \bigg) \\
    &\geq \frac{\mu(\ball(I, \delta))}{2} > 0,
\end{align*}
contradicting our assumption that $I_P(A) = 0$.
\end{proof}

\subsection{Harmonic analysis on the sphere and the Counting Lemma} \label{HarmAnalSphere}

The next thing we need is an analogue of the Counting Lemma in the spherical setting, saying we do not significantly change the count of configurations in a given set $A \subseteq \S^d$ by blurring this set a little.
As in the Euclidean setting, we will use Fourier-analytic methods to prove such a result;
we now give a quick overview of the definitions and results on harmonic analysis we will need for our arguments.

Given an integer $n \geq 0$, we write $\mathscr{H}^{d+1}_n$ for the space of real harmonic polynomials, homogeneous of degree $n$, on $\R^{d+1}$.
That is,
$$\mathscr{H}^{d+1}_n = \bigg\{ f \in \R[x_1, \dots, x_{d+1}]:\, f \text{ homogeneous},\, \deg f = n,\, \sum_{i=1}^{d+1} \frac{\partial^2}{\partial x_i^2} f = 0 \bigg\}.$$
The restriction of the elements of $\mathscr{H}^{d+1}_n$ to $\S^d$ are called \emph{spherical harmonics of degree $n$} on $\S^d$.
If $Y \in \mathscr{H}^{d+1}_n$, note that $Y(x) = \|x\|^n Y(x')$ where $x = \|x\| x'$ and $x' \in \S^d$;
we can then identify $\mathscr{H}^{d+1}_n$ with the space of spherical harmonics of degree $n$, which by a slight (and common) abuse of notation we also denote $\mathscr{H}^{d+1}_n$.

Harmonic polynomials of different degrees are orthogonal with respect to the standard inner product $\langle f,\, g \rangle_{\S^d} := \int_{\S^d} f(x) g(x) \,d\sigma(x)$.
Moreover, it is a well-known fact (see e.g.~\cite[Chapter~1.1]{harmAnalSphere}) that the family of spherical harmonics is dense in $L^2(\S^d)$, and so
$$L^2(\S^d) = \bigoplus_{n=0}^{\infty} \mathscr{H}^{d+1}_n.$$
Denoting by $\proj_{n}: L^2(\S^d) \rightarrow \mathscr{H}^{d+1}_n$ the orthogonal projection onto $\mathscr{H}^{d+1}_n$, what this means is that $f = \sum_{n=0}^{\infty} \proj_n f$ for all $f \in L^2(\S^d)$ (with equality in the $L^2$ sense).
By orthogonality we obtain \emph{Parseval's identity}:
$$\|f\|_2^2 = \sum_{n=0}^{\infty} \|\proj_n f\|_2^2.$$

There is a family $(P_n^d)_{n \geq 0}$ of polynomials on $[-1, 1]$, usually called \emph{ultraspherical} or \emph{Gegenbauer} polynomials, which is associated to this decomposition.
We use the convention that $\deg P_n^d = n$ and $P_n^d(1) = 1$.
These polynomials can be defined via the \emph{addition formula}
\begin{equation} \label{eq:addition}
    P_n^d(x\cdot y) = \frac{1}{\dim \mathscr{H}_n^{d+1}} \sum_{i=1}^{\dim \mathscr{H}_n^{d+1}} Y_i(x) Y_i(y) \quad \text{for all $x, y \in \S^d$},
\end{equation}
where $\{Y_i:\, 1\leq i\leq \dim \mathscr{H}_n^{d+1}\}$ is an (arbitrary) orthonormal basis of $\mathscr{H}_n^{d+1}$.
We refer the reader to Chapter~1.2 of Dai and Xu's book~\cite{harmAnalSphere} for the proof that this formula is independent of the choice of basis, and that it indeed defines a polynomial on $[-1, 1]$.

The next theorem collects several properties of the Gegenbauer polynomials which will be useful for us:

\begin{thm} \label{gegenbauer}
For all integers $d \geq 2$ and $n \geq 0$ the following hold:
\begin{itemize}
    \item[$(i)$] $P^d_n(t) \in [-1, 1]$ for all $t \in [-1, 1]$.
    \item[$(ii)$] The projection operator $\proj_n: L^2(\S^d) \rightarrow \mathscr{H}^{d+1}_n$ is given by
    \begin{equation} \label{projectionidentity}
        \proj_n f(x) = \dim \mathscr{H}_n^{d+1} \int_{\S^d} P_n^d(x \cdot y) f(y) \,d\sigma(y).
    \end{equation}
    \item[$(iii)$] For each fixed $y, z \in \S^d$ we have
    \begin{equation} \label{Pndidentity}
    \int_{\S^d} P_n^d(x \cdot y) P_n^d(x \cdot z) \,d\sigma(x) = \frac{1}{\dim \mathscr{H}_n^{d+1}} P_n^d(y \cdot z).
    \end{equation}
    \item[$(iv)$] For any fixed $\gamma > 0$, $\max_{t \in [-1+\gamma,\, 1-\gamma]} P^d_n(t)$ tends to zero as $n \rightarrow \infty$.
\end{itemize}
\end{thm}

\begin{proof}
The first three items follow easily from the addition formula~\eqref{eq:addition}.
Indeed, fix some orthonormal basis $\{Y_i:\, 1\leq i\leq \dim \mathscr{H}_n^{d+1}\}$ of $\mathscr{H}_n^{d+1}$.
Then
\begin{align*}
    \big|P_n^d(x\cdot y)\big|
    &= \bigg|\int_{\SO(d+1)}  P_n^d(Rx\cdot Ry) \,d\mu(R)\bigg| \\
    &= \frac{1}{\dim \mathscr{H}_n^{d+1}} \bigg|\int_{\SO(d+1)} \sum_{i=1}^{\dim \mathscr{H}_n^{d+1}} Y_i(Rx) Y_i(Ry) \,d\mu(R)\bigg|,
\end{align*}
which by the triangle inequality followed by Cauchy-Schwarz is at most
\begin{align*}
    \frac{1}{\dim \mathscr{H}_n^{d+1}} \sum_{i=1}^{\dim \mathscr{H}_n^{d+1}} &\bigg(\int_{\SO(d+1)} Y_i(Rx)^2 \,d\mu(R)\bigg)^{1/2} \bigg(\int_{\SO(d+1)} Y_i(Ry)^2 \,d\mu(R)\bigg)^{1/2} \\
    &= \frac{1}{\dim \mathscr{H}_n^{d+1}} \sum_{i=1}^{\dim \mathscr{H}_n^{d+1}} \bigg(\int_{\S^d} Y_i(z)^2 \,d\sigma(z)\bigg)
    = 1,
\end{align*}
proving~$(i)$.
Item~$(ii)$ follows from the chain of equalities
\begin{align*}
    \proj_n f(x)
    &= \sum_{i=1}^{\dim \mathscr{H}_n^{d+1}} \bigg(\int_{\S^d} f(y) Y_i(y) \,d\sigma(y)\bigg) Y_i(x) \\
    &= \int_{\S^d} \bigg(\sum_{i=1}^{\dim \mathscr{H}_n^{d+1}} Y_i(x) Y_i(y)\bigg) f(y) \,d\sigma(y) \\
    &= \dim \mathscr{H}_n^{d+1} \int_{\S^d} P_n^d(x \cdot y) f(y) \,d\sigma(y).
\end{align*}
To prove item~$(iii)$, note that
\begin{align*}
    \int_{\S^d} P_n^d(x \cdot y) &P_n^d(x \cdot z) \,d\sigma(x) \\
    &= \frac{1}{(\dim \mathscr{H}_n^{d+1})^2} \int_{\S^d} \sum_{i, j = 1}^{\dim \mathscr{H}_n^{d+1}} Y_i(x) Y_i(y) Y_j(x) Y_j(z) \,d\sigma(x) \\
    &= \frac{1}{(\dim \mathscr{H}_n^{d+1})^2} \sum_{i, j = 1}^{\dim \mathscr{H}_n^{d+1}} \bigg(\int_{\S^d} Y_i(x) Y_j(x)  \,d\sigma(x)\bigg) Y_i(y) Y_j(z) \\
    &= \frac{1}{(\dim \mathscr{H}_n^{d+1})^2} \sum_{i=1}^{\dim \mathscr{H}_n^{d+1}} Y_i(y) Y_i(z),
\end{align*}
which equals the right-hand side of~\eqref{Pndidentity} by definition.

Finally, the last item immediately follows from the more precise asymptotic bound given in \cite[Theorem~8.21.6]{Szego75}.
\end{proof}

We will follow Dunkl~\cite{Dunkl} in defining both the convolution operation on the sphere and the spherical analogue of Fourier coefficients.
For this we will need to break a little the symmetry of the sphere and distinguish an (arbitrary) point $e$ on $\S^d$;
we think of this point as being the north pole.
Write $\mathcal{M}(\S^d; e)$ for the space of Borel regular \emph{zonal measures on $\S^d$ with pole at $e$}, that is, those measures which are invariant under the action of $\Stab^{\SO(\R^{d+1})}(e)$.
We will refer to the elements of $\mathcal{M}(\S^d; e)$ simply as \emph{zonal measures}.

Given a function $f \in L^2(\S^d)$ and a zonal measure $\nu \in \mathcal{M}(\S^d; e)$, we define their \emph{convolution} $f * \nu$ by
$$f * \nu (x) := \int_{\S^d} f(T_x y) \,d\nu(y) \hspace{5mm} \text{for all } x \in \S^d,$$
where $T_x \in \SO(\R^{d+1})$ is an arbitrary element satisfying $T_x e = x$.
It is easy to see that this operation is well-defined, independently of the choice of $T_x$:
if $S_x e = T_x e = x$, then $S_x^{-1} T_x \in \Stab(e)$ and so
$\nu(S_x^{-1} A) = \nu((S_x^{-1} T_x) T_x^{-1} A) = \nu(T_x^{-1} A)$.
The value $f * \nu(x)$ can be thought of as the average of $f$ according to a measure which acts with respect to $x$ as $\nu$ acts with respect to the north pole $e$.

For an integer $n \geq 0$ and a zonal measure $\nu \in \mathcal{M}(\S^d; e)$, we define its \emph{$n$-th Fourier coefficient} $\widehat{\nu}_n$ by
$$\widehat{\nu}_n = \int_{\S^d} P^d_n(e \cdot y) \,d\nu(y).$$
The main property we will need of Fourier coefficients is the following result, which is stated in Dunkl's paper~\cite{Dunkl} and can be proven using a straightforward modification of the methods exposed in Chapter $2$ of Dai and Xu's book~\cite{harmAnalSphere}:

\begin{thm} \label{thm:conv_sphere}
If $f \in L^2(\S^d)$ and $\nu \in \mathcal{M}(\S^d; e)$, then $f * \nu \in L^2(\S^d)$ and
$$\proj_n (f * \nu) = \widehat{\nu}_n \,\proj_n f \hspace{5mm} \text{for all } n \geq 0.$$
\end{thm}

With this we finish our review of harmonic analysis on the sphere, so let us return to our specific problem.
For a given $\delta > 0$, denote by $\scap_{\delta}$ the uniform probability measure on the spherical cap $\Scap(e, \delta)$:
$$\scap_{\delta}(A) = \frac{\sigma(A \cap \Scap(e, \delta))}{\sigma(\Scap(e, \delta))} \hspace{5mm} \text{for all measurable } A \subseteq \S^d.$$
Note that each $\scap_{\delta}$ is a zonal measure.
One immediately checks that
\begin{align*}
    (\widehat{\scap}_{\delta})_n &= \frac{1}{\sigma(\Scap_{\delta})} \int_{\Scap(e, \delta)} P^d_n(e \cdot y) \,d\sigma(y)
\end{align*}
for all $n \geq 0$, and
\begin{align*}
    f * \scap_{\delta}(x) &= \frac{1}{\sigma(\Scap_{\delta})} \int_{\Scap(x, \delta)} f(y) \,d\sigma(y)
\end{align*}
for all $f \in L^2(\S^d)$.
In particular, if $A \subseteq \S^d$ is a measurable set, then $A * \scap_{\delta}(x) = d_{\Scap(x, \delta)}(A)$;
this gives the `blurring' of the spherical sets we shall consider.

\begin{lem} \label{2pointsphere}
For every $d \geq 2$ and $\gamma > 0$, there exists a function $c_{d, \gamma}: (0, 1] \rightarrow \R$ with $\lim_{\delta \rightarrow 0^+} c_{d, \gamma}(\delta) = 0$ such that the following holds:
for all $f, g \in L^2(\S^d)$ and all points $u, v \in \S^d$ with $|u \cdot v| \leq 1 - \gamma$, we have that
$$\bigg| \int_{\SO(\R^{d+1})} f(Ru) \big(g(Rv) - g * \scap_{\delta}(Rv)\big) \,d\mu(R) \bigg| \leq c_{d, \gamma}(\delta)\, \|f\|_2 \|g\|_2.$$
\end{lem}

\begin{proof}
Denote by $\nu_e$ the Haar measure on $\Stab(e)$, and assume without loss of generality that $u$ coincides with the north pole $e$.
By symmetry, the expression we wish to bound may then be written as
\begin{align*}
    \bigg| \int_{\SO(\R^{d+1})} f(Re) h(Rv) \,d\mu(R) \bigg|
    &= \bigg| \int_{\SO(\R^{d+1})} f(Re) \bigg( \int_{\Stab(e)} h(RSv) \,d\nu_e(S) \bigg) \,d\mu(R) \bigg|,
\end{align*}
where $h = g - g * \scap_{\delta}$.

Write $t_0 := e \cdot v$.
Note that, when $S \in \Stab(e)$ is distributed uniformly according to $\nu_e$, the point $Sv$ is uniformly distributed on $\S^{d-1}_{t_0} := \{ y \in \S^d: e \cdot y = t_0 \}$.
Denote by $\sigma_{t_0}^{(d-1)}$ the uniform probability measure on $\S^{d-1}_{t_0}$ (that is, the unique one which is invariant under the action of $\Stab(e)$).
Making the change of variables $y = Sv$, we see that
\begin{equation} \label{convsigma}
    \int_{\Stab(e)} h(RSv) \,d\nu_e(S)
    = \int_{\S^{d-1}_{t_0}} h(Ry) \,d\sigma_{t_0}^{(d-1)}(y)
    = h * \sigma_{t_0}^{(d-1)}(Re).
\end{equation}
The expression we wish to bound is then equal to
$$\bigg| \int_{\SO(\R^{d+1})} f(Re) \,h * \sigma_{t_0}^{(d-1)}(Re) \,d\mu(R) \bigg| = \bigg| \int_{\S^d} f(x) \,h * \sigma_{t_0}^{(d-1)}(x) \,d\sigma(x) \bigg|.$$

Using Parseval's Identity, we can rewrite the right-hand side of the last equality as
\begin{align*}
    \bigg| \sum_{n=0}^{\infty} \int_{\S^d} \proj_n f(x)
    &\,\proj_n (h * \sigma_{t_0}^{(d-1)}) (x) \,d\sigma(x) \bigg| \\
    &\leq \sum_{n=0}^{\infty} \int_{\S^d} |\proj_n f (x)| \, |(\widehat{\sigma}_{t_0}^{(d-1)})_n| \, |\proj_n h (x)| \,d\sigma(x) \\
    &\leq \sum_{n=0}^{\infty} |(\widehat{\sigma}_{t_0}^{(d-1)})_n| \, \|\proj_n f\|_2 \, \|\proj_n h\|_2,
\end{align*}
where we used Theorem~\ref{thm:conv_sphere} and then Cauchy-Schwarz.
As $h = g - g * \scap_{\delta}$, the expression above is equal to
\begin{align*}
    &\sum_{n=0}^{\infty} |(\widehat{\sigma}_{t_0}^{(d-1)})_n| \, |1 - (\widehat{\scap}_{\delta})_n| \, \|\proj_n f\|_2 \, \|\proj_n g\|_2 \\
    &\hspace{5mm} = \sum_{n=0}^{\infty} |P^d_n(t_0)| \bigg| 1 - \frac{1}{\sigma(\Scap_{\delta})} \int_{\Scap(e, \delta)} P^d_n(e \cdot y) \,d\sigma(y) \bigg| \|\proj_n f\|_2 \, \|\proj_n g\|_2.
\end{align*}

Fix some $\varepsilon > 0$.
Since $t_0 \in [-1 + \gamma,\, 1 - \gamma]$ (by hypothesis), from Theorem~\ref{gegenbauer} we obtain that $|P^d_n(t_0)| \leq \varepsilon/2$ holds for all $n \geq N(\varepsilon, \gamma)$, while
$$\bigg| 1 - \frac{1}{\sigma(\Scap_{\delta})} \int_{\Scap(e, \delta)} P^d_n(e \cdot y) \,d\sigma(y) \bigg| \leq \max_{-1 \leq t \leq 1} \big|1 - P^d_n(t)\big| = 2$$
always holds.
Moreover, since each $P^d_n$ is a polynomial satisfying $P^d_n(1) = 1$, we can choose $\delta_0 = \delta_0(\varepsilon, \gamma) > 0$ small enough so that $|1 - P^d_n(e \cdot y)| \leq \varepsilon$ holds whenever $n < N(\varepsilon, \gamma)$ and $y \in \Scap(e, \delta_0)$.
This implies that the last sum is at most
$$\sum_{n=0}^{\infty} \varepsilon \|\proj_n f\|_2 \, \|\proj_n g\|_2 \leq \varepsilon \|f\|_2 \|g\|_2$$
whenever $\delta \leq \delta_0(\varepsilon, \gamma)$, finishing the proof.
\end{proof}

Recall that a spherical configuration $P$ is admissible if it has at most $d$ points and if it is congruent to a set $P' \subset \S^d$ which is linearly independent.
We can now give the spherical counterpart to the Counting Lemma from the last section:

\begin{lem}[Counting Lemma] \label{countingsphere}
For every admissible configuration $P$ on $\S^d$ there exists a function $\eta_P: (0, 1] \rightarrow (0, 1]$ with $\lim_{\delta \rightarrow 0^+} \eta_P(\delta) = 0$ such that the following holds for all measurable sets $A \subseteq \S^d$:
$$|I_P(A) - I_P(A * \scap_{\delta})| \leq \eta_P(\delta) \hspace{5mm} \text{for all } \delta \in (0, 1].$$
Moreover, this upper-bound function $\eta_P$ can be made to hold uniformly over all configurations $P'$ inside a neighborhood of $P$.
\end{lem}

\begin{proof}
Up to congruence, we may assume $P \subset \S^d$.
Similarly to what we did in the Euclidean setting, we will first obtain a uniform upper bound for
$$\bigg| \int_{\SO(\R^{d+1})} f_1(Tv_1) \cdots f_{k-1}(Tv_{k-1}) \big( f_k(Tv_k) - f_k * \scap_{\delta} (Tv_k) \big) \,d\mu(T) \bigg|,$$
valid whenever $0 \leq f_1, \dots, f_k \leq 1$ are measurable functions and $(v_1, v_2, \dots, v_k)$ is a permutation of the points of $P$.

Denote by $G := \Stab^{\SO(\R^{d+1})}(v_1, \dots, v_{k-2})$ the stabilizer of the first $k-2$ points of $P$, and by $H := \Stab^{\SO(\R^{d+1})}(v_1, \dots, v_{k-2}, v_{k-1}) = \Stab^{G}(v_{k-1})$ the stabilizer of the first $k-1$ points of $P$.
We can then bound the expression above by
\begin{equation} \label{whattobound}
    \int_{\SO(\R^{d+1})} \bigg| \int_G f_{k-1}(TSv_{k-1}) \big( f_k(TSv_k) - f_k * \scap_{\delta} (TSv_k) \big) \,d\mu_G(S) \bigg| \,d\mu(T),
\end{equation}
where $\mu_G$ denotes the normalized Haar measure on $G$.

Denote $\ell := d-k+2 \geq 2$.
Since $P$ is non-degenerate, we see that $G \simeq \SO(\R^{\ell+1})$ and that both $Gv_{k-1}$ and $Gv_k$ are spheres of dimension $\ell$.
Morally, we should then be able to apply the last lemma (with $d = \ell$, $f = f_{k-1}(T\cdot)$ and $g = f_k(T\cdot)$) and easily conclude.
However, the convolution in expression (\ref{whattobound}) above happens in $\S^d$, while that on the last lemma would happen in $\S^{\ell}$;
in particular, if $k \geq 3$ so that $\ell < d$, all of the mass on the average defined by the convolution in (\ref{whattobound}) lies \emph{outside} of the $\ell$-dimensional sphere $Gv_k$, so this argument cannot work.
We will have to work harder to conclude.

Note that, since $Gv_k$ is an $\ell$-dimensional sphere while $Hv_k$ is an $(\ell-1)$-dimensional sphere (which happens because $P$ is non-degenerate), it follows that there is a point $\xi \in Gv_k$ which is fixed by $H$;
this point will work as the north pole of $Gv_k$.

It will be more convenient to work on the canonical unit sphere $\S^{\ell}$ instead of the $\ell$-dimensional sphere $Gv_k \subset \S^d$.
We shall then restrict ourselves to the $(\ell+1)$-dimensional affine hyperplane $\mathcal{H}$ determined by $\mathcal{H} \cap \S^d = Gv_k$, and place coordinates on it to identify $\mathcal{H}$ with $\R^{\ell+1}$ and $Gv_k$ with $\S^{\ell}$, noting that $G$ then acts as $\SO(\R^{\ell+1})$.
More formally, let $r > 0$ be the radius of $Gv_k$ in $\R^{d+1}$, so that $Gv_k$ is isometric to $r \S^{\ell}$;
take such an isometry $\psi: Gv_k \rightarrow r \S^{\ell}$, and define $e \in \S^{\ell}$ by $e := \psi(\xi)/r$.
Now we construct a map $\phi: G \rightarrow \SO(\R^{\ell+1})$ defined by
$$\phi(S) \psi(x) = \psi(Sx) \hspace{5mm} \text{for all } x \in Gv_k$$
for each $S \in G$.
It is easy to check that this map is well-defined and gives an isomorphism between $G$ and $\SO(\R^{\ell+1})$ satisfying $\phi(H) = \Stab^{\SO(\R^{\ell+1})}(e)$.

For each fixed $T \in \SO(\R^{d+1})$, define the functions $g_T, h_T: \S^{\ell} \rightarrow [-1, 1]$ by
\begin{align*}
    &g_T(Re) := f_{k-1}(T \phi^{-1}(R) v_{k-1}) \hspace{3mm} \text{and} \\
    &h_T(Re) := f_k(T \phi^{-1}(R) \xi) - f_k * \scap_{\delta}(T \phi^{-1}(R) \xi),
\end{align*}
for all $R \in \SO(\R^{\ell+1})$.
These functions are indeed well-defined on $\S^{\ell}$, since $\Stab^G(v_{k-1}) = \Stab^G(\xi) = \phi^{-1}(\Stab^{\SO(\R^{\ell+1})}(e))$.
Note that $h_T$ can also be written as a function of $x \in \S^{\ell}$ by making use of the isometry $\psi^{-1}:\, r \S^{\ell}\rightarrow Gv_k$:
$$h_T(x) = f_k(T \psi^{-1}(rx)) - f_k * \scap_{\delta}(T \psi^{-1}(rx)).$$

Denote by $u := \psi(v_k)/r$ the point in $\S^{\ell}$ corresponding to $v_k$.
Making the change of variables $R = \phi(S)$, we obtain
\begin{align*}
    \int_G f_{k-1}(TSv_{k-1}) &\big( f_k(TSv_k) - f_k * \scap_{\delta} (TSv_k) \big) \,d\mu_G(S) \\
    &= \int_{\SO(\R^{\ell+1})} g_T(Re) \,h_T(Ru) \,d\mu_{\ell+1}(R) \\
    &= \int_{\SO(\R^{\ell+1})} g_T(Re) \bigg( \int_{\Stab(e)} h_T(RSu) \,d\nu_e(S) \bigg) \,d\mu_{\ell+1}(R),
\end{align*}
where we write $\Stab(e)$ for $\Stab^{\SO(\R^{\ell+1})}(e)$ and $\nu_e$ for its Haar measure.
Working as we did to obtain equation~\eqref{convsigma}, we see that the expression in parenthesis is equal to $h_T * \sigma_{e \cdot u}^{(\ell-1)}(Re)$, where $\sigma_{e \cdot u}^{(\ell-1)}$ is the uniform probability measure on the $(\ell-1)$-sphere $\Stab(e) u = \{y \in \S^{\ell}: e \cdot y = e \cdot u\}$
(and the convolution now takes place in $\S^{\ell}$ with $e$ as the north pole).
Making the change of variables $x = Re$, we then see that the expression above is equal to
$$\int_{\SO(\R^{\ell+1})} g_T(Re) \,h_T * \sigma_{e \cdot u}^{(\ell-1)}(Re) \,d\mu_{\ell+1}(R) = \int_{\S^{\ell}} g_T(x)\, h_T * \sigma_{e \cdot u}^{(\ell-1)}(x) \,d\sigma^{(\ell)}(x).$$
We conclude that the expression (\ref{whattobound}) we wish to bound is equal to
\begin{align*}
    \int_{\SO(\R^{d+1})} \bigg| \int_{\S^{\ell}} g_T(x) \,h_T * \sigma_{e \cdot u}^{(\ell-1)}(x) \,&d\sigma^{(\ell)}(x) \bigg| \,d\mu_{d+1}(T) \\
    &\leq \int_{\SO(\R^{d+1})} \|h_T * \sigma_{e \cdot u}^{(\ell-1)}\|_2 \,d\mu_{d+1}(T) \\
    &\leq \bigg( \int_{\SO(\R^{d+1})} \|h_T * \sigma_{e \cdot u}^{(\ell-1)}\|_2^2 \,d\mu_{d+1}(T) \bigg)^{1/2},
\end{align*}
where we applied Cauchy-Schwarz twice.

Let us now compute $e \cdot u$, which will be necessary for bounding $\|h_T * \sigma_{e \cdot u}^{(\ell-1)}\|_2^2$.
From the identity
$$\|re - ru\|_{\mathbb{R}^{\ell+1}}^2 = \|\psi^{-1}(re) - \psi^{-1}(ru)\|_{\mathbb{R}^{d+1}}^2 = \|\xi - v_k\|_{\mathbb{R}^{d+1}}^2,$$
we conclude that $r^2(2 - 2\, e \cdot u) = 2 - 2\, \xi \cdot v_k$, and so
$$e \cdot u = (\xi \cdot v_k - (1 - r^2))/r^2 \notin \{-1,\, 1\}$$
depends only on the ordering $(v_1, \dots, v_k)$ of $P$ and not on our later choices.
(Note that this value depends continuously on the points $v_1, \dots, v_k$, and so is bounded away from $\{-1,\, 1\}$ uniformly over all configurations $P'$ close enough to $P$.)

Now fix an arbitrary $\varepsilon > 0$.
By Parseval's Identity and Theorem~\ref{thm:conv_sphere}, we have that
\begin{align*}
    \|h_T * \sigma_{e \cdot u}^{(\ell-1)}\|_2^2 &= \sum_{n=0}^{\infty} \|\proj_n (h_T * \sigma_{e \cdot u}^{(\ell-1)})\|_2^2 \\
    &= \sum_{n=0}^{\infty} |(\widehat{\sigma}_{e \cdot u}^{(\ell-1)})_n|^2 \,\|\proj_n h_T\|_2^2 \\
    &= \sum_{n=0}^{\infty} P^{\ell}_n(e \cdot u)^2 \,\|\proj_n h_T\|_2^2.
\end{align*}
Since $e \cdot u \notin \{-1,\, 1\}$ is a constant depending only on $P$, by Theorem~\ref{gegenbauer} there exists $N = N(\varepsilon, P) \in \mathbb{N}$ such that $|P^{\ell}_n(e \cdot u)| \leq \varepsilon$ for all $n > N$.
(By that same theorem, this value of $N$ can be made robust to small perturbations of the value $e \cdot u$, which corresponds to small perturbations of the configuration $P$.)
Using that $|P^{\ell}_n(t)| \leq 1$ for all $-1 \leq t \leq 1$, we conclude that
$$\|h_T * \sigma_{e \cdot u}^{(\ell-1)}\|_2^2 \leq \sum_{n=0}^{N} \|\proj_n h_T\|_2^2 + \sum_{n > N} \varepsilon^2 \|\proj_n h_T\|_2^2.$$
The second term on the right-hand side of the inequality above is upper bounded by $\varepsilon^2 \|h_T\|_2^2 \leq \varepsilon^2$, so let us concentrate on the first term.

By identities (\ref{projectionidentity}) and (\ref{Pndidentity}), we have
\begin{align*}
    &\|\proj_n h_T\|_2^2 = \int_{\S^{\ell}} \bigg( \dim \mathscr{H}_n^{\ell+1} \int_{\S^{\ell}} h_T(y) P_n^{\ell}(x \cdot y) \,d\sigma(y) \bigg)^2 \,d\sigma(x) \\
    &\hspace{1cm} = (\dim \mathscr{H}_n^{\ell+1})^2 \int_{\S^{\ell}} \int_{\S^{\ell}} h_T(y) h_T(z) \bigg( \int_{\S^{\ell}} P_n^{\ell}(x \cdot y) P_n^{\ell}(x \cdot z) \,d\sigma(x) \bigg) \,d\sigma(y) \,d\sigma(z) \\
    &\hspace{1cm} = \dim \mathscr{H}_n^{\ell+1} \int_{\S^{\ell}} \int_{\S^{\ell}} h_T(y) h_T(z) \,P_n^{\ell}(y \cdot z) \,d\sigma(y) \,d\sigma(z).
\end{align*}
Since $|P_n^{\ell}(y \cdot z)| \leq 1$ for all $y, z \in \S^{\ell}$, we conclude that
\begin{align*}
    \int_{\SO(d+1)} &\|\proj_n h_T\|_2^2 \,d\mu_{d+1}(T) \\
    &= \dim \mathscr{H}_n^{\ell+1} \int_{\S^{\ell}} \int_{\S^{\ell}} \bigg( \int_{\SO(d+1)} h_T(y) h_T(z) \,d\mu_{d+1}(T) \bigg) P_n^{\ell}(y \cdot z) \,d\sigma(y) \,d\sigma(z) \\
    &\leq \dim \mathscr{H}_n^{\ell+1} \int_{\S^{\ell}} \int_{\S^{\ell}} \bigg| \int_{\SO(d+1)} h_T(y) h_T(z) \,d\mu_{d+1}(T) \bigg| \,d\sigma(y) \,d\sigma(z).
\end{align*}

We now divide this last double integral on the sphere into two parts, depending on whether or not $y \cdot z$ is close to the extremal points $1$ or $-1$.
Thus, for some  parameter $0 < \gamma < 1$ to be chosen later, we write the double integral as
\begin{align*}
    \int_{\S^{\ell}} \int_{\S^{\ell}} &\bigg| \int_{\SO(d+1)} h_T(y) h_T(z) \,d\mu_{d+1}(T) \bigg| \mathbbm{1} \big\{ |y \cdot z| > 1 - \gamma \big\} \,d\sigma(y) \,d\sigma(z) \\
    &+ \int_{\S^{\ell}} \int_{\S^{\ell}} \bigg| \int_{\SO(d+1)} h_T(y) h_T(z) \,d\mu_{d+1}(T) \bigg| \mathbbm{1} \big\{ |y \cdot z| \leq 1 - \gamma \big\} \,d\sigma(y) \,d\sigma(z).
\end{align*}
Since $-1 \leq h_T \leq 1$, the first term is at most
$$2 \int_{\S^{\ell}} \int_{\S^{\ell}} \mathbbm{1}\{y \cdot z > 1 - \gamma \} \,d\sigma(y) \,d\sigma(z) = 2 \sigma^{(\ell)}\big(\Scap_{\S^{\ell}}(e, \sqrt{2\gamma})\big).$$
To bound the second term, note that for fixed $y, z \in \S^{\ell}$ we have
\begin{align*}
    \int_{\SO(d+1)} &h_T(y) h_T(z) \,d\mu_{d+1}(T) \\
    &= \int_{\SO(d+1)} \big(f_k(T\widetilde{y}) - f_k * \scap_{\delta}(T\widetilde{y}) \big) \big(f_k(T\widetilde{z}) - f_k * \scap_{\delta}(T\widetilde{z}) \big) \,d\mu_{d+1}(T),
\end{align*}
where $\widetilde{y} := \psi^{-1}(ry)$ and $\widetilde{z} := \psi^{-1}(rz)$.
Moreover, we have
$$\|ry - rz\|_{\mathbb{R}^{\ell+1}}^2 = \|\widetilde{y} - \widetilde{z}\|_{\mathbb{R}^{d+1}}^2 \implies \widetilde{y} \cdot \widetilde{z} = 1 - r^2(1 - y \cdot z);$$
thus, whenever $|y \cdot z| \leq 1-\gamma$, we have $|\widetilde{y} \cdot \widetilde{z}| \leq 1-r^2\gamma$.
Using Lemma~\ref{2pointsphere} (with $f = f_k - f_k * \scap_{\delta}$, $g = f_k$ and $\gamma$ substituted by $r^2\gamma$) we conclude that the second term is bounded by $c_{d, r^2\gamma}(\delta)$.

Taking stock of everything, we obtain
\begin{align*}
    \int_{\SO(d+1)} &\|h_T * \sigma_{e \cdot u}^{(\ell-1)}\|_2^2 \,d\mu_{d+1}(T) \\
    &\leq \varepsilon^2 + \sum_{n=0}^N \dim \mathscr{H}_n^{\ell+1} \big( 2 \sigma^{(\ell)} \big(\Scap_{\S^{\ell}}(e, \sqrt{2\gamma})\big) + c_{d, r^2\gamma}(\delta) \big)
\end{align*}
for any $0 < \gamma < 1$.
Choosing $\gamma$ small enough depending on $\ell$, $\varepsilon$ and $N$, and then choosing $\delta$ small enough depending on $d$, $r^2\gamma$, $\varepsilon$ and $N$ (so ultimately only on $\varepsilon$ and $P$), we can bound the right-hand side above by $4\varepsilon^2$;
the expression (\ref{whattobound}) is then bounded by $2 \varepsilon$ in this case.

For such small values of $\delta$, we thus conclude from our telescoping sum trick (explained in Section~\ref{FourierSection}) that $|I_P(A) - I_P(A * \scap_{\delta})| \leq 2k\varepsilon$, proving the desired inequality since $\varepsilon > 0$ is arbitrary.
The claim that the upper bound can be made uniform inside some neighborhood of $P$ follows from analyzing our proof.
\end{proof}

We remark that the proof of the Counting Lemma given above is the only place where we explicitly make use of the assumption that a spherical configuration is admissible.
This assumption, however, will get inherited by all later results which make use of the Counting Lemma in their proofs.

\subsection{Continuity properties of the counting function}

Following the same script as in the Euclidean setting, we now consider other ways in which the counting function is robust to small perturbations.

It is again easy to show, using our telescoping sum trick, that $I_P$ is continuous in $L^{\infty}(\S^d)$ (and even in $L^{|P|}(\S^d)$) for all spherical configurations.
When the configuration considered is admissible, we obtain also the following significantly stronger continuity property of $I_P$ when restricting to bounded functions:

\begin{lem}[Weak$^*$ continuity] \label{lem:weak*cont_sphere}
If $P$ is an admissible configuration on $\S^d$, then $I_P$ is weak$^*$ continuous on the unit ball of $L^{\infty}(\S^d)$.
\end{lem}

\begin{proof}
Denote the closed unit ball of $L^{\infty}(\S^d)$ by $\ball_{\infty}$, and let $(f_i)_{i \geq 1} \subset \ball_{\infty}$ be a sequence weak$^*$ converging to $f \in \ball_{\infty}$.
It will suffice to show that $\left(I_P(f_i)\right)_{i \geq 1}$ converges to $I_P(f)$.

Note that, for every $x \in \S^d$, $\delta > 0$, we have
\begin{align*}
    f_i * \scap_{\delta}(x) = &\frac{1}{\sigma(\Scap_{\delta})} \int_{\Scap(x, \delta)} f_i(y) \,d\sigma(y) \\
    &\xrightarrow{i \rightarrow \infty} \frac{1}{\sigma(\Scap_{\delta})} \int_{\Scap(x, \delta)} f(y) \,d\sigma(y) = f * \scap_{\delta}(x).
\end{align*}
Since $f * \scap_{\delta}$ and each $f_i * \scap_{\delta}$ are Lipschitz with the same constant (depending only on $\delta$) and $\S^d$ is compact, this easily implies that
$$\|f_i * \scap_{\delta} - f * \scap_{\delta}\|_{\infty} \rightarrow 0 \hspace{3mm} \text{as } i \rightarrow \infty.$$
In particular, we conclude
$\lim_{i \rightarrow \infty} I_P(f_i * \scap_{\delta}) = I_P(f * \scap_{\delta})$.

Since $P$ is admissible, by the spherical Counting Lemma we have
$$|I_P(f * \scap_{\delta}) - I_P(f)| \leq \eta_P(\delta) \hspace{3mm} \text{and} \hspace{3mm} |I_P(f_i * \scap_{\delta}) - I_P(f_i)| \leq \eta_P(\delta) \hspace{3mm} \text{for all } i \geq 1.$$
Choosing $i_0(\delta) \geq 1$ sufficiently large so that
$$|I_P(f_i * \scap_{\delta}) - I_P(f * \scap_{\delta})| \leq \eta_P(\delta) \hspace{5mm} \text{for all } i \geq i_0(\delta),$$
we conclude that
\begin{align*}
    |I_P(f) - I_P(f_i)| &\leq |I_P(f) - I_P(f * \scap_{\delta})| + |I_P(f * \scap_{\delta}) - I_P(f_i * \scap_{\delta})| \\
    & \hspace{1cm}
    + |I_P(f_i * \scap_{\delta}) - I_P(f_i)| \\
    &\leq 3 \eta_P(\delta) \hspace{5mm} \text{for all } i \geq i_0(\delta).
\end{align*}
Since $\delta > 0$ is arbitrary and $\eta_P(\delta) \rightarrow 0$ as $\delta \rightarrow 0$, this finishes the proof.
\end{proof}

Given some spherical configuration $P = \{v_1, \dots, v_k\} \subset \R^{d+1}$, let us write $\ball(P, r) \subset (\R^{d+1})^k$ for the ball of radius $r$ centered on $P$, where the distance from $P$ to $Q = \{u_1, \dots, u_k\}$ is given by
$$\|Q - P\|_{\infty} := \min_{\sigma\in \mathfrak{S}_k} \max_{1 \leq i \leq k} \|u_i - v_{\sigma(i)}\|.$$
If $P$ is an admissible spherical configuration, note that all configurations inside a small enough ball centered on $P$ will also be admissible.

We will later need an equicontinuity property for the family of counting functions $P \mapsto I_P(A)$, over all measurable sets $A \subseteq \S^d$;
this is given in the following lemma:

\begin{lem}[Equicontinuity] \label{lem:equicont_sphere}
For every admissible $P \subset \S^d$ and every $\varepsilon > 0$, there exists $\delta > 0$ such that
$$|I_Q(A) - I_P(A)| \leq \varepsilon \quad \text{for all $Q \in \ball(P, \delta)$, $A \subseteq \S^d$.}$$
\end{lem}

\begin{proof}
We will use the fact that the function $\eta_P$ obtained in the Counting Lemma can be made uniform inside a small ball centered on $P$.
In other words, there is $r > 0$ and a function $\eta_P': (0, 1] \rightarrow (0, 1]$ with $\lim_{t \rightarrow 0} \eta_P'(t) = 0$ such that
$$|I_Q(A) - I_Q(A * \scap_{\rho})| \leq \eta_P'(\rho) \quad \text{for all $Q \in \ball(P, r)$, $A \subseteq \S^d$.}$$

Now, for a given $\rho > 0$ and all $0 < \delta < \rho$, we see from the triangle inequality that
$$\|x - y\| \leq \delta \implies \Scap(x,\, \rho - \delta) \subset \Scap(x,\, \rho) \cap \Scap(y,\, \rho),$$
and so $\sigma \big(\Scap(x,\, \rho) \setminus \Scap(y,\, \rho)\big) \leq \sigma(\Scap_{\rho}) - \sigma(\Scap_{\rho - \delta})$.
This implies that, for any set $A \subseteq \S^d$ and any $x,\, y \in \S^d$ with $\|x - y\| \leq \delta$, we have
\begin{align*}
    |A * \scap_{\rho}(x) - A * \scap_{\rho}(y)|
    &= \frac{\big| \sigma(A \cap \Scap(x,\, \rho)) - \sigma(A \cap \Scap(y,\, \rho)) \big|}{\sigma(\Scap_{\rho})} \\
    &\leq \frac{\sigma \big(\Scap(x,\, \rho) \setminus \Scap(y,\, \rho)\big)}{\sigma(\Scap_{\rho})} \\
    &\leq \frac{\sigma(\Scap_{\rho}) - \sigma(\Scap_{\rho - \delta})}{\sigma(\Scap_{\rho})}.
\end{align*}
By our telescoping sum trick, whenever $\|Q - P\|_{\infty} \leq \delta$ we conclude that
$$|I_Q(A * \scap_{\rho}) - I_P(A * \scap_{\rho})| \leq k \frac{\sigma(\Scap_{\rho}) - \sigma(\Scap_{\rho - \delta})}{\sigma(\Scap_{\rho})}.$$

Take $\rho > 0$ small enough so that $\eta_P'(\rho) \leq \varepsilon/3$, and for this value of $\rho$ take $0 < \delta < r$ small enough so that $\sigma(\Scap_{\rho - \delta}) \geq (1 - \varepsilon/3k) \,\sigma(\Scap_{\rho})$.
Then, for any $Q \in \ball(P,\, \delta)$ and any measurable set $A \subseteq \S^d$, we have
\begin{align*}
    |I_Q(A) - I_P(A)| &\leq |I_Q(A) - I_Q(A * \scap_{\rho})| + |I_Q(A * \scap_{\rho}) - I_P(A * \scap_{\rho})| \\
    & \hspace{1cm}
    + |I_P(A * \scap_{\rho}) - I_P(A)| \\
    &\leq \eta_P'(\rho) + k \frac{\sigma(\Scap_{\rho}) - \sigma(\Scap_{\rho - \delta})}{\sigma(\Scap_{\rho})} + \eta_P'(\rho) \\
    &\leq \frac{\varepsilon}{3} + k \frac{\varepsilon}{3k} + \frac{\varepsilon}{3}
    = \varepsilon,
\end{align*}
as wished.
\end{proof}

\subsection{The spherical Supersaturation Theorem} \label{SupersatSectionSphere}

Having proven that the counting function for admissible spherical configurations is robust to various kinds of small perturbations, we next show that it also satisfies a useful \emph{supersaturation property}.

This is the second main technical tool we need to study the independence density in the spherical setting, and due to the fact that the unit sphere is compact both its statement and proof are somewhat simpler than in the Euclidean space setting.

\begin{thm} [Supersaturation Theorem] \label{thm:supersat_sphere}
For every admissible configuration $P$ on $\S^d$ and every $\varepsilon > 0$ there exists a constant $c(\varepsilon) > 0$ such that the following holds:
if $A \subseteq \S^d$ satisfies $\sigma(A) \geq \m_{\S^d}(P) + \varepsilon$, then $I_P(A) \geq c(\varepsilon)$.
\end{thm}

\begin{proof}
Suppose for contradiction that the result is false;
then there exist some $\varepsilon > 0$ and some sequence $(A_i)_{i \geq 1}$ of sets, each of density at least $\m_{\S^d}(P) + \varepsilon$, which satisfy $\lim_{i \rightarrow \infty} I_P(A_i) = 0$.

Note that the unit ball $\ball_{\infty}$ of $L^{\infty}(\S^d)$ is weak$^*$ compact, and also metrizable in this topology (see~\cite[Chapter~2.6]{Megginson98}).
By possibly restricting to a subsequence, we may then assume that $(A_i)_{i \geq 1}$ converges in the weak$^*$ topology of $L^{\infty}(\S^d)$;
let us denote its limit by $A \in \ball_{\infty}$.
It is clear that $0 \leq A \leq 1$ almost everywhere, and
$\int_{\S^d} A(x) \,d\sigma(x) = \lim_{i \rightarrow \infty} \sigma(A_i) \geq \m_{\S^d}(P) + \varepsilon$.
By weak$^*$ continuity of $I_P$ (Lemma~\ref{lem:weak*cont_sphere}), we also have $I_P(A) = \lim_{i \rightarrow \infty} I_P(A_i) = 0$.

Now let $B := \{x \in \S^d:\, A(x) \geq \varepsilon\}$.
Since
$$\varepsilon B(x) \leq A(x) < \varepsilon + B(x) \quad \text{for a.e. } x \in \S^d,$$
we conclude that $I_P(B) \leq \varepsilon^{-|P|} I_P(A) = 0$ and
$$\sigma(B) > \int_{\S^d} A(x) \,d\sigma(x) - \varepsilon \geq \m_{\S^d}(P).$$
But this set $B$ contradicts Lemma~\ref{lem:zero_meas_sphere}, finishing the proof.
\end{proof}

It will be useful to also introduce a spherical analogue of the zooming-out operator,
which acts on measurable spherical sets and represents the points on the sphere around which the considered set has a somewhat high density.
Given quantities $\delta$, $\gamma > 0$, we denote by $\mathcal{Z}_{\delta}(\gamma)$ the operator which takes a measurable set $A \subseteq \S^d$ to the set
$$\mathcal{Z}_{\delta}(\gamma)[A] := \big\{x \in \S^d:\, d_{\Scap(x, \delta)}(A) \geq \gamma\big\}.$$
The most important property of the zooming-out operator is the following result:

\begin{cor} \label{cor:supersat_sphere}
For every admissible configuration $P$ on $\S^d$ and every $\varepsilon > 0$, there exists $\delta_0 > 0$ such that the following holds for all $\delta \leq \delta_0$:
if $A \subseteq \S^d$ satisfies
$$\sigma \big(\mathcal{Z}_{\delta}(\varepsilon)[A]\big) \geq \m_{\S^d}(P) + \varepsilon,$$
then $A$ contains a congruent copy of $P$.
\end{cor}

\begin{proof}
By the Supersaturation Theorem, we know that
$$\sigma \big(\mathcal{Z}_{\delta}(\varepsilon)[A]\big) \geq \m_{\S^d}(P) + \varepsilon \implies I_P \big(\mathcal{Z}_{\delta}(\varepsilon)[A]\big) \geq c(\varepsilon)$$
holds for all $\delta > 0$.
By the Counting Lemma, we then have
\begin{align*}
    I_P(A) \geq I_P(A * \scap_{\delta}) - \eta_P(\delta)
    &\geq \varepsilon^{|P|} I_P \big(\mathcal{Z}_{\delta}(\varepsilon)[A]\big) - \eta_P(\delta) \\
    &\geq \varepsilon^{|P|} c(\varepsilon) - \eta_P(\delta).
\end{align*}
Since $\eta_P(\delta) \rightarrow 0$ as $\delta \rightarrow 0$, there is some $\delta_0 > 0$ such that for all $\delta \leq \delta_0$ we can conclude $I_P(A) > 0$;
this implies that $A$ contains a copy of $P$.
\end{proof}

\subsection{From the sphere to spherical caps}

We must now tackle the problem of obtaining a relationship between the independence density $\m_{\S^d}(P)$ of a given configuration $P \subset \S^d$ and its spherical cap version $\m_{\Scap(x, \rho)}(P)$, as this will be needed later.

In the Euclidean setting this was very easy to do (see Lemma~\ref{lem:bounds_space}), using the fact that we can tessellate $\R^d$ with cubes $Q(x, R)$ of any given side length $R > 0$.
This is no longer the case in the spherical setting, as it is impossible to completely cover $\S^d$ using non-overlapping spherical caps of some given radius;
in fact, this cannot be done even approximately if we require the radii of the spherical caps to be the same (as we did with the side length of the cubes in $\R^d$).

We will then need to use a much weaker `almost-covering' result, saying that we can cover almost all of the sphere by using finitely many non-overlapping spherical caps with possibly different radii.
Such a collection of disjoint spherical caps is called a \emph{cap packing}.
For technical reasons, we will also want the radii of the caps in this packing to be arbitrarily small.

\begin{lem} \label{densepacking}
For every $\varepsilon > 0$ there is a finite cap packing
$$\mathcal{P} = \big\{\Scap(x_i, \rho_i):\, 1 \leq i \leq N\big\}$$
of $\S^d$ with density $\sigma(\mathcal{P}) > 1 - \varepsilon$ and with radii $\rho_i \leq \varepsilon$ for all $1 \leq i \leq N$.
\end{lem}

\begin{proof}
We will use the same notation for both a collection of caps and the set of points on $\S^d$ which belong to (at least) one of these caps.
The desired packing $\mathcal{P}$ will be constructed in several steps, starting with $\mathcal{P}_0 := \{\Scap(e, \varepsilon)\}$.

Now suppose $\mathcal{P}_{i-1}$ has already been constructed (and is finite) for some $i \geq 1$, and let us construct $\mathcal{P}_i$.
Define
$$\mathcal{C}_i := \big\{ \Scap \big( x,\, \min\{\varepsilon,\, \dist(x, \mathcal{P}_{i-1})\} \big):\, x\in \S^d \setminus \mathcal{P}_{i-1} \big\},$$
and note that $\mathcal{C}_i$ is a covering of $\S^d \setminus \mathcal{P}_{i-1}$ by caps of positive radii (since $\mathcal{P}_{i-1}$ is closed on $\S^d$).
By the Vitali Covering Lemma, there is a countable subcollection
$$\mathcal{Q}_i = \bigcup_{j=1}^{\infty} \{\Scap(x_j, r_j)\} \subset \mathcal{C}_i$$
of \emph{disjoint} caps in $\mathcal{C}_i$ such that
$\S^d \setminus \mathcal{P}_{i-1} \subseteq \bigcup_{j=1}^{\infty} \Scap(x_j, 5r_j)$.
In particular
$$1 - \sigma(\mathcal{P}_{i-1}) = \sigma(\S^d \setminus \mathcal{P}_{i-1})
\leq \sum_{j=1}^{\infty} \sigma(\Scap(x_j, 5r_j)) \leq K_d \,\sigma(\mathcal{Q}_i),$$
where we denote $K_d := \sup_{r > 0} \sigma(\Scap_{5r})/\sigma(\Scap_r) < \infty$.
Taking $N_i \in \mathbb{N}$ such that
$$\sum_{j=1}^{N_i} \sigma(\Scap(x_j, r_j)) \geq \sigma(\mathcal{Q}_i) - \frac{1 - \sigma(\mathcal{P}_{i-1})}{2 K_d},$$
we see that $\mathcal{P}_i' := \{\Scap(x_j, r_j):\, 1 \leq j \leq N_i\} \subset \S^d \setminus \mathcal{P}_{i-1}$ satisfies
$$\sigma(\mathcal{P}_i') \geq \frac{1 - \sigma(\mathcal{P}_{i-1})}{2 K_d}.$$

Now set $\mathcal{P}_i := \mathcal{P}_{i-1} \cup \mathcal{P}_i'$;
this is a finite cap packing with
\begin{align*}
    1 - \sigma(\mathcal{P}_i) = 1 - \sigma(\mathcal{P}_{i-1}) - \sigma(\mathcal{P}_i')
    &\leq (1 - \sigma(\mathcal{P}_{i-1})) \bigg(1 - \frac{1}{2 K_d}\bigg) \\
    &\leq (1 - \sigma(\Scap_{\varepsilon})) \bigg(1 - \frac{1}{2  K_d}\bigg)^i
\end{align*}
(where the last inequality follows by induction).
Taking $n \geq 1$ large enough so that $(1 - \sigma(\Scap_{\varepsilon})) \Big(1 - \frac{1}{2 K_d}\Big)^n < \varepsilon$, we see that $\mathcal{P} := \mathcal{P}_n$ satisfies all requirements.
\end{proof}

We can now obtain our analogue of Lemma~\ref{lem:bounds_space}, relating the two versions of independence density in the spherical setting:

\begin{lem} \label{spheretocap}
For every $\varepsilon > 0$, $\rho > 0$ there exists $t_0 > 0$ such that the following holds whenever $P_1, \dots, P_n \subset \S^d$ have diameter at most $t_0$:
$$\big| \m_{\Scap(x, \rho)}(P_1,\, \dots,\, P_n) - \m_{\S^d}(P_1,\, \dots,\, P_n) \big| < \varepsilon.$$
\end{lem}

\begin{proof}
If $A \subset \S^d$ is a set which avoids $P_1, \dots, P_n$, then for every $x \in \S^d$ the set $A \cap \Scap(x, \rho) \subseteq \Scap(x, \rho)$ also avoids $P_1, \dots, P_n$.
Since $\mathbb{E}_{x \in \S^d}[d_{\Scap(x, \rho)}(A)] = \sigma(A)$, there must exist some $x \in \S^d$ such that
$$d_{\Scap(x, \rho)}(A \cap \Scap(x, \rho)) = d_{\Scap(x, \rho)}(A) \geq \sigma(A);$$
optimizing over $A$ we conclude that
$\m_{\Scap(x, \rho)}(P_1,\, \dots,\, P_n) \geq \m_{\S^d}(P_1,\, \dots,\, P_n)$.

For the opposite direction, let $\gamma \leq \varepsilon/4$ be small enough so that $\sigma(\Scap_{\rho + \gamma}) \leq (1 + \varepsilon/4)\, \sigma(\Scap_{\rho})$.
By Lemma~\ref{densepacking}, we know there is a cap packing
$$\mathcal{P} = \{\Scap(x_i, \rho_i):\, 1 \leq i \leq N\}$$
of $\S^d$ with $\sigma(\mathcal{P}) \geq 1 - \gamma$ and $0 < \rho_1, \dots, \rho_N \leq \gamma$.
Now let $t_0 > 0$ be small enough so that $\sigma(\Scap_{\rho_i - 2t_0}) \geq (1 - \varepsilon/4)\, \sigma(\Scap_{\rho_i})$ for all $1 \leq i \leq N$;
note that $t_0$ will ultimately depend only on $\varepsilon$ and $\rho$.

Fixing any configurations $P_1, \dots, P_n \subset \S^d$ of diameter at most $t_0$, let $A \subset \Scap(x, \rho)$ be a set which avoids all of them.
We shall construct a set $\widetilde{A} \subset \S^d$ which also avoids $P_1, \dots, P_n$, and which satisfies $\sigma(\widetilde{A}) > d_{\Scap(x, \rho)}(A) - \varepsilon$;
this will finish the proof.

For each $1 \leq i \leq N$, denote $\widetilde{\rho}_i := \rho_i - 2t_0 < \gamma$.
We have that
\begin{align*}
    \sigma(A) &= \int_{\S^d} d_{\Scap(y, \widetilde{\rho}_i)}(A) \,d\sigma(y) \\
    &= \int_{\Scap(x,\, \rho + \widetilde{\rho}_i)} d_{\Scap(y, \widetilde{\rho}_i)}(A) \,d\sigma(y) \\
    &\leq \int_{\Scap(x,\, \rho)} d_{\Scap(y, \widetilde{\rho}_i)}(A) \,d\sigma(y) + \sigma(\Scap_{\rho + \widetilde{\rho}_i}) - \sigma(\Scap_{\rho}).
\end{align*}
Since $\widetilde{\rho}_i < \gamma$, dividing by $\sigma(\Scap_{\rho})$ we obtain
\begin{align*}
    \mathbb{E}_{y \in \Scap(x, \rho)} \left[ d_{\Scap(y, \widetilde{\rho}_i)}(A) \right]
    &\geq \frac{\sigma(A)}{\sigma(\Scap_{\rho})} - \frac{\sigma(\Scap_{\rho + \widetilde{\rho}_i}) - \sigma(\Scap_{\rho})}{\sigma(\Scap_{\rho})} \\
    &> d_{\Scap(x, \rho)}(A) - \frac{\varepsilon}{4}.
\end{align*}
There must then exist $y_i \in \Scap(x,\, \rho)$ for which $d_{\Scap(y_i, \widetilde{\rho}_i)}(A) > d_{\Scap(x, \rho)}(A) - \varepsilon/4$;
fix one such $y_i$ for each $1 \leq i \leq N$, and let $T_{y_i \rightarrow x_i} \in \SO(\R^{d+1})$ be any rotation taking $y_i$ to $x_i$ (and thus taking $\Scap(y_i,\, \widetilde{\rho}_i)$ to $\Scap(x_i,\, \widetilde{\rho}_i)$).

We claim that the set
$$\widetilde{A} := \bigcup_{i=1}^N T_{y_i \rightarrow x_i} (A \cap \Scap(y_i,\, \widetilde{\rho}_i))$$
satisfies our requirements.
Indeed, we have
\begin{align*}
    \sigma(\widetilde{A}) = \sum_{i=1}^N \sigma(A \cap \Scap(y_i,\, \widetilde{\rho}_i))
    &= \sum_{i=1}^N d_{\Scap(y_i, \widetilde{\rho}_i)}(A) \cdot \sigma(\Scap_{\widetilde{\rho}_i}) \\
    &> \sum_{i=1}^N \left( d_{\Scap(x, \rho)}(A) - \frac{\varepsilon}{4} \right) \cdot \left( 1 - \frac{\varepsilon}{4} \right) \sigma(\Scap_{\rho_i}) \\
    &\geq \left( d_{\Scap(x, \rho)}(A) - \frac{\varepsilon}{2} \right) \sigma(\mathcal{P}) \\
    &> d_{\Scap(x, \rho)}(A) - \varepsilon.
\end{align*}
Moreover, since $\diam(P_j) \leq t_0$ and the caps $\Scap(x_i,\, \widetilde{\rho}_i)$ are (at least) $2t_0$-distant from each other, we see that any copy of $P_j$ in $\widetilde{A} \subset \bigcup_{i=1}^N \Scap(x_i,\, \widetilde{\rho}_i)$ must be entirely contained in one of the the caps $\Scap(x_i,\, \widetilde{\rho}_i)$.
But then it should also be contained (after rotation by $T_{y_i \rightarrow x_i}^{-1}$) in $A \cap \Scap(y_i,\, \widetilde{\rho}_i)$;
this shows that $\widetilde{A}$ does not contain copies of $P_j$ for any $1 \leq j \leq N$, since $A$ does not, and we are done.
\end{proof}

\subsection{Results on the spherical independence density}

We are finally ready to start a more detailed study of the independence density parameter in the spherical setting.

We start by providing a general lower bound on the independence density of several different configurations in terms of their individual independence densities:

\begin{lem}[Supermultiplicativity] \label{lem:supermult_sphere}
For all configurations $P_1, \dots, P_n$ on $\S^d$, we have
$$\m_{\S^d}(P_1,\, \dots,\, P_n) \geq \prod_{i=1}^n \m_{\S^d}(P_i).$$
\end{lem}

\begin{proof}
Choose, for each $1 \leq i \leq n$, a set $A_i \subset \S^d$ which avoids configuration $P_i$.
By taking independent rotations $R_i A_i$ of each set $A_i$, we see that
\begin{align*}
    \mathbb{E}_{R_1, \dots, R_n \in \SO(\R^{d+1})} \bigg[\sigma \bigg(\bigcap_{i=1}^n R_i A_i\bigg)\bigg]
    &= \int_{\S^d} \prod_{i=1}^n \mathbb{E}_{R_i \in \SO(\R^{d+1})} \big[A_i(R_i^{-1} x)\big] \,d\sigma(x) \\
    &= \prod_{i=1}^n \sigma(A_i).
\end{align*}
There must then exist $R_1, \dots, R_n \in \SO(\R^{d+1})$ for which
$$\sigma \bigg(\bigcap_{i=1}^n R_i A_i\bigg) \geq \prod_{i=1}^n \sigma(A_i).$$
Since $\bigcap_{i=1}^n R_i A_i$ avoids all configurations $P_1, \dots, P_n$ and the sets $A_1, \dots, A_n$ were chosen arbitrarily, the result follows.
\end{proof}

Using supersaturation, we can show that this lower bound is essentially tight when the configurations considered are all admissible and each one is at a different size scale.
Intuitively, this happens because the constraints of avoiding each of these configurations will act at distinct scales and thus not correlate with each other.

\begin{thm}[Asymptotic independence] \label{thm:main_sphere}
For every admissible configurations $P_1, \dots$, $P_n$ on $\S^d$ and every $0 < \varepsilon \leq 1$ there is a positive increasing function $f: (0, 1] \rightarrow (0, 1]$ such that the following holds:
whenever $0 < t_1, \dots, t_n \leq 1$ satisfy $t_{i+1} \leq f(t_i)$ for $1 \leq i < n$, we have
$$\bigg| \m_{\S^d}(t_1 P_1,\, \dots,\, t_n P_n) - \prod_{i=1}^n \m_{\S^d}(t_i P_i) \bigg| \leq \varepsilon.$$
\end{thm}

\begin{proof}
We have already seen that $\m_{\S^d}(t_1 P_1,\, \dots,\, t_n P_n) \geq \prod_{i=1}^n \m_{\S^d}(t_i P_i)$, so it suffices to show that $\m_{\S^d}(t_1 P_1,\, \dots,\, t_n P_n) \leq \prod_{i=1}^n \m_{\S^d}(t_i P_i) + \varepsilon$ for suitably separated $t_1, \dots, t_n \leq 1$.
We will do so by induction on $n$, with the base case $n=1$ being trivial (and taking $f \equiv 1$, say).

Suppose then $n \geq 2$ and we have already proven the result for $n-1$ configurations.
Let $\tilde{f}: (0, 1] \rightarrow (0, 1]$ be the function promised by the theorem applied to the $n-1$ configurations $P_2, \dots, P_n$ and with accuracy $\varepsilon$,
so that whenever $0 < t_2 \leq 1$ and $0 < t_{j+1} \leq \tilde{f}(t_j)$ for each $2 \leq j < n$ we have
$$\m_{\S^d}(t_2 P_2,\, \dots,\, t_n P_n) \leq \prod_{j=2}^n \m_{\S^d}(t_j P_j) + \varepsilon.$$

By the corollary to the Supersaturation Theorem (Corollary~\ref{cor:supersat_sphere}), for all $0 < t_1 \leq 1$ there is $\delta_0 = \delta_0(\varepsilon;\, t_1 P_1) > 0$ such that
$$\sigma \big(\mathcal{Z}_{\delta_0}(\varepsilon)[A]\big) \geq \m_{\S^d}(t_1 P_1) + \varepsilon \implies A \text{ contains a copy of } t_1 P_1.$$
Applying Lemma~\ref{spheretocap} with radius $\rho = \delta_0$, we see there is $t_0 = t_0(\varepsilon,\, \delta_0) > 0$ for which
$$\m_{\Scap(x, \delta_0)}(t_2 P_2,\, \dots,\, t_n P_n) \leq \m_{\S^d}(t_2 P_2,\, \dots,\, t_n P_n) + \varepsilon$$
holds whenever $0 < t_2, \dots, t_n \leq t_0/2$.

Let now $0 < t_1, \dots, t_n \leq 1$ be numbers satisfying
$$t_2 \leq t_0(\varepsilon,\, \delta_0(\varepsilon;\, t_1 P_1))/2 \hspace{3mm} \text{and} \hspace{3mm} t_{j+1} \leq \tilde{f}(t_j) \text{ for all } 2 \leq j < n.$$
If $A \subset \S^d$ does not contain copies of $t_1 P_1, \dots, t_n P_n$, then by the preceding discussion we must have $\sigma \big(\mathcal{Z}_{\delta_0}(\varepsilon)[A]\big) < \m_{\S^d}(t_1 P_1) + \varepsilon$ and, for all $x \in \S^d$,
\begin{align*}
    d_{\Scap(x, \delta_0)}(A) \leq \m_{\Scap(x, \delta_0)}(t_2 P_2,\, \dots,\, t_n P_n)
    &\leq \m_{\S^d}(t_2 P_2,\, \dots,\, t_n P_n) + \varepsilon \\
    &\leq \prod_{j=2}^n \m_{\S^d}(t_j P_j) + 2\varepsilon.
\end{align*}
This means that, inside caps $\Scap(x, \delta_0)$ of radius $\delta_0$, $A$ has density less than $\varepsilon$ (when $x \notin \mathcal{Z}_{\delta_0}(\varepsilon)[A]$) except on a set of measure at most $\m_{\S^d}(t_1 P_1) + \varepsilon$, when it instead has density at most $\prod_{j=2}^n \m_{\S^d}(t_j P_j) + 2\varepsilon$.
Taking averages, we conclude that
\begin{align*}
    \sigma(A) &= \mathbb{E}_{x \in \S^d} \big[d_{\Scap(x, \delta)}(A)\big] \\
    &\leq \varepsilon + \big( \m_{\S^d}(t_1 P_1) + \varepsilon \big) \bigg( \prod_{j=2}^n \m_{\S^d}(t_j P_j) + 2\varepsilon \bigg) \\
    &\leq 6\varepsilon + \prod_{i=1}^n \m_{\S^d}(t_i P_i).
\end{align*}
It thus suffices to take the function $f: (0, 1] \rightarrow (0, 1]$ given by
$$f(t) = \min \bigg\{ \tilde{f}(t),\, \frac{t_0 \big(\varepsilon/6,\, \delta_0(\varepsilon/6;\, t P_1)\big)}{2} \bigg\}$$
to conclude the induction.
\end{proof}

Note that this result provides a partial answer to the analogue of question (Q1) in the spherical setting:
if $P$ is admissible, then $\m_{\S^d}(t_1 P,\, t_2 P,\, \dots,\, t_n P)$ decays exponentially with $n$ as the ratios $t_{j+1}/t_j$ between consecutive scales go to zero
(recall from Lemma~\ref{lem:reassuring} that $\m_{\S^d}(t P)$ is bounded away from both zero and one for $0 < t \leq 1$).

By considering an infinite sequence of `counterexamples' as we did in our proof of Bourgain's Theorem (Theorem~\ref{thmbourgain}), we immediately obtain from Theorem~\ref{thm:main_sphere} the following result:

\begin{cor} \label{cor:Bourgain_sphere}
Let $P \subset \S^d$ be an admissible configuration.
If $A \subseteq \S^d$ has positive measure, then there is some number $t_0 > 0$ such that $A$ contains a congruent copy of $t P$ for all $t \leq t_0$.
\end{cor}

This corollary can be seen as the counterpart to Bourgain's Theorem in the spherical setting, where it impossible to consider arbitrarily large
dilates.
(The equivalent result of containing all sufficiently small dilates of a configuration in the Euclidean setting also holds with the same proof.)

We will next prove that the independence density function $P \mapsto \m_{\S^d}(P)$ is continuous on the set of admissible configurations on $\S^d$.
Before doing so, it is interesting to note that a similar result does \emph{not} hold for two-point configurations on the unit circle $\S^1$ (which can be seen as the very first instance of non-admissible configurations).
Indeed, it was shown by DeCorte and Pikhurko~\cite{Evan} that $\m_{\S^1}(\{u, v\})$ is \emph{discontinuous} at a configuration $\{u, v\} \subset \S^1$ whenever the arc length between $u$ and $v$ is a rational multiple of $2\pi$ with odd denominator.

\begin{thm}[Continuity of the independence density] \label{thm:cont_of_m_sphere}
For any $n \geq 1$, the function $(P_1, \dots, P_n) \mapsto \m_{\S^d}(P_1, \dots, P_n)$ is continuous on the set of $n$ admissible spherical configurations.
\end{thm}

\begin{proof}
For simplicity of exposition we will prove the result in the case of only one forbidden configuration, but the general case follows from the same argument.

Fix some $\varepsilon > 0$ and some admissible configuration $P$ on $\S^d$, and let $c(\varepsilon) > 0$ be the constant promised by the Supersaturation Theorem (Theorem \ref{thm:supersat_sphere}).
By equicontinuity (Lemma~\ref{lem:equicont_sphere}) there exists $\delta > 0$ such that
$$|I_Q(A) - I_P(A)| \leq \varepsilon \quad \text{for all $Q \in \ball(P, \delta)$, $A \subseteq \S^d$.}$$
Suppose $Q \in \ball(P, \delta)$ and $A \subset \S^d$ is a measurable set avoiding $Q$;
we must then have $I_P(A) \leq c(\varepsilon)$, and so $\sigma(A) \leq \m_{\S^d}(P) + \varepsilon$.
Optimizing over $A$, we conclude that $\m_{\S^d}(Q) \leq \m_{\S^d}(P) + \varepsilon$ whenever $Q \in \ball(P, \delta)$.

Now write $P = \{v_1, \dots, v_k\}$, and consider the function
$g_P: (\S^d)^k \times \SO(\R^{d+1}) \rightarrow \R$ given by
$$g_P(x_1, \dots, x_k, T) := \sum_{i=1}^k \|x_i - T v_i\|.$$
Note that this function is continuous and that $\min_{T \in \SO(\R^{d+1})} g_P(x_1, \dots, x_k, T) = 0$ if and only if $(x_1, \dots, x_k)$ is congruent to $(v_1, \dots, v_k)$.

By inner regularity, we can find a compact set $A \subset \S^d$ which avoids $P$ and has measure $\sigma(A) \geq \m_{\S^d}(P) - \varepsilon$.
The continuous function $g_P$ attains a minimum on the compact set $A^k \times \SO(\R^{d+1})$;
denote this minimum by $\gamma$, and note that $\gamma > 0$ since $A$ avoids $P$.
Let us show that $A$ also avoids $Q$, for all $Q \in \ball(P, \gamma/2k)$.
Indeed, writing $Q = \{u_1, \dots, u_k\}$ (with the labels chosen so as to minimize their distance to the corresponding points of $P$), for any points $x_1, \dots, x_k \in A$ and any $T \in \SO(\R^{d+1})$ we have that
\begin{align*}
    \sum_{i=1}^k \|x_i - T u_i\| &\geq \sum_{i=1}^k \big|\|x_i - T v_i\| - \|T u_i - T v_i\|\big| \\
    &\geq g_P(x_1, \dots, x_k, T) - k \|Q - P\|_{\infty},
\end{align*}
which is at least $\gamma/2 > 0$ if $\|Q - P\|_{\infty} \leq \gamma/2k$.
For such configurations we then obtain
$$\m_{\S^d}(Q) \geq \sigma(A) \geq \m_{\S^d}(P) - \varepsilon.$$

We conclude that $|\m_{\S^d}(Q) - \m_{\S^d}(P)| \leq \varepsilon$ whenever $\|Q - P\|_{\infty} \leq \min \{\delta, \gamma/2k\}$, finishing the proof.
\end{proof}

As our definition of the independence density $\m_{\S^d}(P)$ involved a supremum over all $P$-avoiding measurable sets $A \subseteq \S^d$, it is not immediately clear whether there actually exists a measurable $P$-avoiding set attaining this extremal value of density.
In fact, such a result is \emph{false} in the case where $d = 1$ and we are considering two-point configurations $\{u, v\} \subset \S^1$:
if the length of the arc between $u$ and $v$ is \emph{not} a rational multiple of $\pi$, it was shown by Sz\'ekely~\cite{Szekely} that $\m_{\S^1}(\{u, v\}) = 1/2$ but there is no $\{u, v\}$-avoiding measurable set of density $1/2$.

We will now show that extremizer sets exist whenever the configuration we are forbidding is admissible.
Note that the result also holds (with essentially unchanged proof) when forbidding several admissible configurations;
this generalizes to higher-order configurations a theorem of DeCorte and Pikhurko~\cite{Evan} for forbidden distances on the sphere.

\begin{thm}[Existence of extremizers]
If $P \subset \S^{d}$ is an admissible configuration, then there exists a $P$-avoiding measurable set $A \subseteq \S^d$ attaining $\sigma(A) = \m_{\S^d}(P)$.
\end{thm}

\begin{proof}
Let $A_1, A_2, \dots \subseteq \S^d$ be a sequence of $P$-avoiding measurable sets satisfying $\lim_{i \rightarrow \infty} \sigma(A_i) = \m_{\S^d}(P)$.
By passing to a subsequence if necessary, we may assume that $(A_i)_{i \geq 1}$ converges to some function $A \in \ball_{\infty}$ in the weak$^*$ topology of $L^{\infty}(\S^d)$.
We shall prove two things:
\begin{itemize}
    \item[$(i)$] the limit function $A$ is $\{0,\, 1\}$-valued almost everywhere, so we can identify it with its support $\text{supp} \,A$;
    \item[$(ii)$] after possibly modifying it on a zero-measure set, this set $A$ will avoid $P$.
\end{itemize}
With these two results we will be done, since $\sigma(A) = \lim_{i \rightarrow \infty} \sigma(A_i) = \m_{\S^d}(P)$.

By weak$^*$ convergence we know that $0 \leq A \leq 1$ almost everywhere, and by weak$^*$ continuity (Lemma~\ref{lem:weak*cont_sphere}) we also have
$I_P(A) = \lim_{i \rightarrow \infty} I_P(A_i) = 0$.
From this we easily conclude that $I_P(\text{supp} \,A) = 0$, and also
\begin{equation} \label{suppA>A}
    \sigma(\text{supp} \,A) = \int_{\S^d} \text{supp} \,A(x) \,d\sigma(x) \geq \int_{\S^d} A(x) \,d\sigma(x) = \m_{\S^d}(P).
\end{equation}
But Lemma~\ref{lem:zero_meas_sphere} implies that $\sigma(\text{supp} \,A) \leq \m_{\S^d}(P)$, which by (\ref{suppA>A}) and the fact that $0 \leq A \leq 1$ can only happen if $A = \text{supp} \,A$ almost everywhere.
This proves $(i)$.

Identifying $A$ with its support and using that $I_P(A) = 0$, Lemma~\ref{lem:zero_meas_sphere} implies we can remove a zero-measure subset of $A$ in order to remove all copies of $P$.
This proves item $(ii)$ and finishes the proof of the theorem.
\end{proof}

To conclude, let us make explicit what we can say about the possible independence densities when forbidding $n$ distinct contractions of an admissible configuration $P$;
due to lack of dilation invariance in the spherical setting, characterizing these values in terms of simpler quantities is much harder than it is in the Euclidean setting.

Denote
$\mathcal{M}_n^{\S^d}(P) := \big\{\m_{\S^d}(t_1 P,\, t_2 P,\, \dots,\, t_n P):\, 0 < t_1 < t_2 < \dots < t_n \leq 1\big\}$.
Due to continuity of $\m_{\S^d}$ (Theorem~\ref{thm:cont_of_m_sphere}) this set is an interval, and its upper extremity is $\sup_{0 < t \leq 1} \m_{\S^d}(t P)$.
By supermultiplicativity (Lemma~\ref{lem:supermult_sphere}) the lower extremity of $\mathcal{M}_n^{\S^d}(P)$ is at least $\inf_{0 < t \leq 1} \m_{\S^d}(t P)^n$, and by asymptotic independence (Theorem~\ref{thm:main_sphere}) it can be at most
$\inf_{0 < t \leq 1} \m_{\S^d}(t P) \cdot \liminf_{t \rightarrow 0} \m_{\S^d}(t P)^{n-1}$.

\section{Concluding remarks and open problems} \label{Conclusion}

Our results leave open the question of what happens when the configurations we forbid are \emph{not} admissible.
There are two different reasons for a given configuration (either on the space or on the sphere) to not be admissible, so let us examine them separately.

The fist reason is that $P$ is degenerate, meaning that its points are affinely dependent if we are on $\R^d$ or linearly dependent if we are on $\S^d$.
In the Euclidean setting, Bourgain~\cite{Bourgain} showed an example of sets $A_d \subset \R^d$ (for each $d \geq 2$) which have positive density but which avoid arbitrarily large dilates of a degenerate three-point configuration of the form $\{-v, 0, v\}$.
These sets then show that the conclusion of Bourgain's Theorem (and thus also the conclusion of our Theorem~\ref{thm:main_space}) is false for this degenerate configuration.

This counterexample was later generalized by Graham~\cite{Graham}, who showed that a result like Bourgain's Theorem \emph{can only hold} if $P$ is contained on the surface of some sphere of finite radius (as is always the case when $P$ is non-degenerate).
In fact, Graham's result implies (for instance) that $$\m_{\R^d}\big(P,\, \sqrt{3}P,\, \sqrt{5}P,\, \sqrt{7}P,\, \dots\big) > 0$$
whenever $P \subset \R^d$ is nonspherical, that is, not contained on the surface of any sphere.
Some kind of non-degeneracy hypothesis is thus necessary both for Bourgain's result and for our
Theorem~\ref{thm:main_space}.\footnote{We believe that the same is true for the spherical analogue of Theorem~\ref{thm:main_space}, namely Theorem~\ref{thm:main_sphere}, though we do not know of a counterexample.}

It is interesting to note, however, that more recent results of Ziegler~\cite{ZieglerFirst, Ziegler} (generalizing a theorem of Furstenberg, Katznelson and Weiss~\cite{FurstKatzWeiss} for three-point configurations) show that every set $A \subseteq \R^d$ of positive upper density is \emph{arbitrarily close} to containing all large enough dilates of any finite configuration $P \subset \R^d$.
More precisely, denoting by $A_{\delta}$ the set of all points at distance at most $\delta$ from the set $A$, Ziegler proved the following:

\begin{thm}
Let $A \subseteq \R^d$ be a set of positive upper density and $P \subset \R^d$ be a finite set.
Then there exists $t_0 > 0$ such that, for any $t \geq t_0$ and any $\delta > 0$, the set $A_{\delta}$ contains a configuration congruent to $t P$.
\end{thm}

The proof of this theorem is ergodic theoretic in nature, making essential use of deep and difficult results regarding nilflows and the characteristic factors of non-conventional ergodic averages.
It unfortunately does not seem to follow from our methods.

Let us now turn to the second reason for a configuration $P$ on $\R^d$ or $\S^d$ to be non-admissible, namely that it contains $d+1$ points (if it has more than $d+1$ points then it is obviously degenerate).
In this case we cannot apply the same strategy we used to prove the Counting Lemmas, and it is not clear whether they or the analogues of Bourgain's Theorem are true.
We conjecture that they are whenever $d \geq 2$, so that we can remove the cardinality condition from the statement of Bourgain's result and of our `asymptotic independence' Theorem~\ref{thm:main_space} and Theorem~\ref{thm:main_sphere}.

In particular, let us make more explicit the simplest case of this conjecture, which is an obvious question left open since the results of Bourgain and of Furstenberg, Katznelson and Weiss:

\begin{conj}
Let $A \subset \R^2$ be a set of positive upper density and let $u, v, w \in \R^2$ be non-collinear points.
Then there exists $t_0 > 0$ such that for any $t \geq t_0$ the set $A$ contains a configuration congruent to $\{t u, t v, t w\}$.
\end{conj}

Another question we ask is related to a suspected compatibility condition between the Euclidean and spherical settings.
Since $\S^d$ resembles $\R^d$ at small scales, it seems geometrically intuitive that $\m_{\S^d}(t P)$ should get increasingly close to $\m_{\R^d}(P)$ as $t \rightarrow 0$ whenever $P$ is a contractible configuration on $\S^d$.
(It is easy to show that a configuration $P \subset \S^d$ is contractible if and only if it is contained in a $d$-dimensional affine subspace, so we can embed it in $\R^d$.)
We ask whether this intuition is indeed correct, i.e. is it true that $\lim_{t \rightarrow 0} \m_{\S^d}(t P) = \m_{\R^d}(P)$ for all contractible configurations $P \subset \S^d$?

In a more combinatorial perspective, we wish to know whether an analogue of the Hypergraph Removal Lemma holds for forbidden geometrical configurations.
In intuitive terms, the question we ask is whether a measurable set $A$ (either on $\R^d$ or on $\S^d$) which contains `few' copies of some given configuration $P$ can be made $P$-avoiding by removing only `a few' of its
points.\footnote{On the unit sphere, this property would more formally read:
whenever $A \subseteq \S^d$ satisfies $I_P(A) \leq \varepsilon$, there is a subset $E \subset A$ of measure $\sigma(E) \leq o_{\varepsilon \rightarrow 0}(1)$ such that $A \setminus E$ avoids $P$ (where $o_{\varepsilon \rightarrow 0}(1)$ denotes a quantity that goes to zero as $\varepsilon \rightarrow 0$).
Similarly in the Euclidean setting.}
Such a result would then explain geometrical sets having few copies of $P$ as those which are close to a set avoiding this configuration, and it trivially implies the corresponding Supersaturation Theorem;
note that this is a quantitative and stronger version of our zero-measure removal Lemmas~\ref{lem:zero_meas_space} and~\ref{lem:zero_meas_sphere}.

Finally, it would be very interesting to have a way of obtaining good upper bounds for the independence densities of a given configuration or family of configurations.
There are several papers (see~\cite{densityBachoc, completelyPositive} and the references therein) which consider this question in the case of a single two-point configuration, drawing on powerful methods from the theory of conic optimization and representation theory, and it is already quite challenging in this simplest case.
Oliveira and Vallentin also considered the case of several forbidden two-point configurations in Euclidean space~\cite{OliveiraV2010} and in arbitrary compact, connected, rank-one symmetric spaces~\cite{OliveiraV2013};
they use linear and semidefinite programming methods to
prove that the independence density of $n$ distinct two-point configurations decays exponentially with $n$ if their sizes are sufficiently far
apart.\footnote{This holds only if the real dimension of the considered space is 2 or higher.
Oliveira and Vallentin \cite[Section~2]{OliveiraV2013} provide counterexamples to the corresponding result in spaces of real dimension 1.}

We believe that the study of the independence density for higher-order configurations in the optimization setting is also worthwhile, since they serve as model problems for symmetric optimization problems depending on higher-order relations and might prove very fruitful in new methods developed.

\section*{Acknowledgements}

The author would like to thank Fernando de Oliveira Filho, Lucas Slot and Frank Vallentin for many helpful discussions.
We also thank the anonymous reviewer, Fernando de Oliveira Filho, and Frank Vallentin for several suggestions which improved the presentation of this paper.

This work was carried out while the author was a PhD student at the University of Cologne.
It was supported by the European Union’s EU Framework Programme for Research and Innovation Horizon 2020 under the Marie Skłodowska-Curie Actions Grant Agreement No 764759 (MINOA), and by the Dutch Research Council (NWO) as part of the NETWORKS programme (grant no. 024.002.003).

\bibliography{config}

\providecommand{\noopsort}[1]{}
\begin{thebibliography}{10}

\bibitem{densityBachoc}
{\sc C.~Bachoc, A.~Passuello, and A.~Thiery}, {\em The density of sets avoiding
  distance 1 in {E}uclidean space}, Discrete Comput. Geom., 53 (2015),
  pp.~783--808.

\bibitem{Bourgain}
{\sc J.~Bourgain}, {\em A {S}zemer\'{e}di type theorem for sets of positive
  density in {${\bf R}^k$}}, Israel J. Math., 54 (1986), pp.~307--316.

\bibitem{Bukh}
{\sc B.~Bukh}, {\em Measurable sets with excluded distances}, Geom. Funct.
  Anal., 18 (2008), pp.~668--697.

\bibitem{quasirandomness}
{\sc D.~Castro-Silva}, {\em Quasirandomness in additive groups and
  hypergraphs}, arXiv preprint arXiv:2107.01463,  (2021).

\bibitem{harmAnalSphere}
{\sc F.~Dai and Y.~Xu}, {\em Approximation theory and harmonic analysis on
  spheres and balls}, Springer Monographs in Mathematics, Springer, New York,
  2013.

\bibitem{completelyPositive}
{\sc E.~DeCorte, F.~M. de~Oliveira~Filho, and F.~Vallentin}, {\em Complete
  positivity and distance-avoiding sets}, Mathematical Programming,  (2020).

\bibitem{Evan}
{\sc E.~DeCorte and O.~Pikhurko}, {\em Spherical sets avoiding a prescribed set
  of angles}, Int. Math. Res. Not. IMRN,  (2016), pp.~6095--6117.

\bibitem{Dunkl}
{\sc C.~F. Dunkl}, {\em Operators and harmonic analysis on the sphere}, Trans.
  Amer. Math. Soc., 125 (1966), pp.~250--263.

\bibitem{ErdosProblems}
{\sc P.~Erd\H{o}s}, {\em Problems and results in combinatorial geometry}, in
  Discrete geometry and convexity ({N}ew {Y}ork, 1982), vol.~440 of Ann. New
  York Acad. Sci., New York Acad. Sci., New York, 1985, pp.~1--11.

\bibitem{Supersaturation}
{\sc P.~Erd\H{o}s and M.~Simonovits}, {\em Supersaturated graphs and
  hypergraphs}, Combinatorica, 3 (1983), pp.~181--192.

\bibitem{IntersectionTheorems}
{\sc P.~Frankl and R.~M. Wilson}, {\em Intersection theorems with geometric
  consequences}, Combinatorica, 1 (1981), pp.~357--368.

\bibitem{FurstKatzWeiss}
{\sc H.~Furstenberg, Y.~Katznelson, and B.~Weiss}, {\em Ergodic theory and
  configurations in sets of positive density}, in Mathematics of {R}amsey
  theory, vol.~5 of Algorithms Combin., Springer, Berlin, 1990, pp.~184--198.

\bibitem{Graham}
{\sc R.~L. Graham}, {\em Recent trends in {E}uclidean {R}amsey theory},
  Discrete Math., 136 (1994), pp.~119--127.

\bibitem{Mattila95}
{\sc P.~Mattila}, {\em Geometry of sets and measures in {E}uclidean spaces},
  vol.~44 of Cambridge Studies in Advanced Mathematics, Cambridge University
  Press, Cambridge, 1995.
\newblock Fractals and rectifiability.

\bibitem{Megginson98}
{\sc R.~E. Megginson}, {\em An introduction to {B}anach space theory}, vol.~183
  of Graduate Texts in Mathematics, Springer-Verlag, New York, 1998.

\bibitem{OliveiraV2010}
{\sc F.~M. \noopsort{O}de Oliveira~Filho and F.~Vallentin}, {\em Fourier
  analysis, linear programming, and densities of distance avoiding sets in
  {$\Bbb R^n$}}, J. Eur. Math. Soc. (JEMS), 12 (2010), pp.~1417--1428.

\bibitem{OliveiraV2013}
\leavevmode\vrule height 2pt depth -1.6pt width 23pt, {\em A quantitative
  version of {S}teinhaus' theorem for compact, connected, rank-one symmetric
  spaces}, Geom. Dedicata, 167 (2013), pp.~295--307.

\bibitem{RegularityGraphs}
{\sc V.~R\"{o}dl and M.~Schacht}, {\em Regularity lemmas for graphs}, in Fete
  of combinatorics and computer science, vol.~20 of Bolyai Soc. Math. Stud.,
  J\'{a}nos Bolyai Math. Soc., Budapest, 2010, pp.~287--325.

\bibitem{Stein93}
{\sc E.~M. Stein}, {\em Harmonic analysis: real-variable methods,
  orthogonality, and oscillatory integrals}, vol.~43 of Princeton Mathematical
  Series, Princeton University Press, Princeton, NJ, 1993.
\newblock With the assistance of Timothy S. Murphy, Monographs in Harmonic
  Analysis, III.

\bibitem{SW71}
{\sc E.~M. Stein and G.~Weiss}, {\em Introduction to {F}ourier analysis on
  {E}uclidean spaces}, Princeton Mathematical Series, No. 32, Princeton
  University Press, Princeton, N.J., 1971.

\bibitem{Szego75}
{\sc G.~Szeg\H{o}}, {\em Orthogonal polynomials}, American Mathematical Society
  Colloquium Publications, Vol. XXIII, American Mathematical Society,
  Providence, R.I., fourth~ed., 1975.

\bibitem{RemarksChromatic}
{\sc L.~A. Sz\'{e}kely}, {\em Remarks on the chromatic number of geometric
  graphs}, in Graphs and other combinatorial topics ({P}rague, 1982), vol.~59
  of Teubner-Texte Math., Teubner, Leipzig, 1983, pp.~312--315.

\bibitem{Szekely}
\leavevmode\vrule height 2pt depth -1.6pt width 23pt, {\em Measurable chromatic
  number of geometric graphs and sets without some distances in {E}uclidean
  space}, Combinatorica, 4 (1984), pp.~213--218.

\bibitem{ZieglerFirst}
{\sc T.~Ziegler}, {\em An application of ergodic theory to a problem in
  geometric {R}amsey theory}, Israel J. Math., 114 (1999), pp.~271--288.

\bibitem{Ziegler}
\leavevmode\vrule height 2pt depth -1.6pt width 23pt, {\em Nilfactors of
  {$\mathbb R^m$}-actions and configurations in sets of positive upper density
  in {$\mathbb R^m$}}, J. Anal. Math., 99 (2006), pp.~249--266.

\end{thebibliography}
\bibliographystyle{siam}

\Addresses

\end{document}